\documentclass[a4paper, 12pt]{amsart}

\usepackage{amsmath,amssymb,enumitem,verbatim,stmaryrd,xcolor,microtype,graphicx}
\usepackage[T1]{fontenc}
\usepackage[utf8]{inputenc}
\usepackage[english]{babel}
\usepackage[top=3.4cm,bottom=3.4cm,left=3.2cm,right=3.2cm]{geometry}
\usepackage[bookmarksdepth=3,linktoc=page,colorlinks,linkcolor={red!80!black},citecolor={red!80!black},urlcolor={blue!80!black}]{hyperref}

\usepackage{tikz}\usetikzlibrary{matrix,arrows,decorations.markings}
\usepackage{tikz-cd}
\pgfrealjobname{hypergeometry}

\usepackage[mathcal]{euscript}                               
\usepackage{mathptmx}                                        
\usepackage{etoolbox}\makeatletter\patchcmd{\@startsection}{\@afterindenttrue}{\@afterindentfalse}{}{}\makeatother    
\patchcmd{\section}{\scshape}{\bfseries}{}{}\makeatletter\renewcommand{\@secnumfont}{\bfseries}\makeatother           
\usepackage[backgroundcolor=yellow,linecolor=yellow,textsize=footnotesize]{todonotes} \makeatletter \providecommand \@dotsep{5} \def\listtodoname{List of Todos} \def\listoftodos{\@starttoc{tdo}\listtodoname} \makeatother 

\theoremstyle{plain}
\newtheorem{thm}{Theorem}[section]
\newtheorem{cor}[thm]{Corollary}
\newtheorem{lemma}[thm]{Lemma}
\newtheorem{prop}[thm]{Proposition}
\newtheorem{thmA}{Theorem}  

\theoremstyle{definition}

\newtheorem{rem}[thm]{Remark}
\newtheorem*{rem*}{Remark}
\newtheorem{ex}[thm]{Example}
\newtheorem{advantage}{Property}

\usepackage{etoolbox}
\makeatletter
\patchcmd{\@startsection}{\@afterindenttrue}{\@afterindentfalse}{}{}             
\patchcmd{\part}{\bfseries}{\bfseries\LARGE}{}{}
\patchcmd{\section}{\scshape}{\bfseries}{}{}\renewcommand{\@secnumfont}{\bfseries} 
\patchcmd{\@settitle}{\uppercasenonmath\@title}{\large}{}{}
\patchcmd{\@setauthors}{\MakeUppercase}{}{}{}
\addto{\captionsenglish}{} 
\addto{\captionsenglish}{} 
\addto{\captionsenglish}{} 
\makeatother

\usepackage{fancyhdr}

\DeclareRobustCommand{\gobblefour}[4]{}    
\DeclareSymbolFont{sfoperators}{OT1}{bch}{m}{n} \DeclareSymbolFontAlphabet{\mathsf}{sfoperators} \makeatletter\def\operator@font{\mathgroup\symsfoperators}\makeatother 
\DeclareSymbolFont{cmletters}{OML}{cmm}{m}{it}              
\DeclareSymbolFont{cmsymbols}{OMS}{cmsy}{m}{n}
\DeclareSymbolFont{cmlargesymbols}{OMX}{cmex}{m}{n}
\DeclareMathSymbol{\myjmath}{\mathord}{cmletters}{"7C}     \let\jmath\myjmath 
\DeclareMathSymbol{\myamalg}{\mathbin}{cmsymbols}{"71}     
\DeclareMathSymbol{\mycoprod}{\mathop}{cmlargesymbols}{"60}\let\coprod\mycoprod
\DeclareMathSymbol{\myalpha}{\mathord}{cmletters}{"0B}     \let\alpha\myalpha 
\DeclareMathSymbol{\mybeta}{\mathord}{cmletters}{"0C}      \let\beta\mybeta
\DeclareMathSymbol{\mygamma}{\mathord}{cmletters}{"0D}     \let\gamma\mygamma
\DeclareMathSymbol{\mydelta}{\mathord}{cmletters}{"0E}     \let\delta\mydelta
\DeclareMathSymbol{\myepsilon}{\mathord}{cmletters}{"0F}   \let\epsilon\myepsilon
\DeclareMathSymbol{\myzeta}{\mathord}{cmletters}{"10}      \let\zeta\myzeta
\DeclareMathSymbol{\myeta}{\mathord}{cmletters}{"11}       \let\eta\myeta
\DeclareMathSymbol{\mytheta}{\mathord}{cmletters}{"12}     \let\theta\mytheta
\DeclareMathSymbol{\myiota}{\mathord}{cmletters}{"13}      \let\iota\myiota
\DeclareMathSymbol{\mykappa}{\mathord}{cmletters}{"14}     \let\kappa\mykappa
\DeclareMathSymbol{\mylambda}{\mathord}{cmletters}{"15}    \let\lambda\mylambda
\DeclareMathSymbol{\mymu}{\mathord}{cmletters}{"16}        \let\mu\mymu
\DeclareMathSymbol{\mynu}{\mathord}{cmletters}{"17}        \let\nu\mynu
\DeclareMathSymbol{\myxi}{\mathord}{cmletters}{"18}        \let\xi\myxi
\DeclareMathSymbol{\mypi}{\mathord}{cmletters}{"19}        \let\pi\mypi
\DeclareMathSymbol{\myrho}{\mathord}{cmletters}{"1A}       \let\rho\myrho
\DeclareMathSymbol{\mysigma}{\mathord}{cmletters}{"1B}     \let\sigma\mysigma
\DeclareMathSymbol{\mytau}{\mathord}{cmletters}{"1C}       \let\tau\mytau
\DeclareMathSymbol{\myupsilon}{\mathord}{cmletters}{"1D}   \let\upsilon\myupsilon
\DeclareMathSymbol{\myphi}{\mathord}{cmletters}{"1E}       \let\phi\myphi
\DeclareMathSymbol{\mychi}{\mathord}{cmletters}{"1F}       \let\chi\mychi
\DeclareMathSymbol{\mypsi}{\mathord}{cmletters}{"20}       \let\psi\mypsi
\DeclareMathSymbol{\myomega}{\mathord}{cmletters}{"21}     \let\omega\myomega
\DeclareMathSymbol{\myvarepsilon}{\mathord}{cmletters}{"22}\let\varepsilon\myvarepsilon
\DeclareMathSymbol{\myvartheta}{\mathord}{cmletters}{"23}  \let\vartheta\myvartheta
\DeclareMathSymbol{\myvarpi}{\mathord}{cmletters}{"24}     \let\varpi\myvarpi
\DeclareMathSymbol{\myvarrho}{\mathord}{cmletters}{"25}    \let\varrho\myvarrho
\DeclareMathSymbol{\myvarsigma}{\mathord}{cmletters}{"26}  \let\varsigma\myvarsigma
\DeclareMathSymbol{\myvarphi}{\mathord}{cmletters}{"27}    \let\varphi\myvarphi

\DeclareMathOperator{\Spec}{Spec}

\DeclareMathOperator{\Hom}{Hom}

\DeclareMathOperator{\OBAff}{{OBAff}}

\DeclareMathOperator{\Trop}{{Trop}}
\DeclareMathOperator{\Bend}{{Bend}}
\DeclareMathOperator{\bend}{{bend}}

\DeclareMathOperator{\OBlpr}{{OBlpr}}

\DeclareMathOperator{\SRings}{{SRings}}

\DeclareMathOperator{\sign}{{sign}}

\newcommand\A{{\mathbb A}}
\newcommand\B{{\mathbb B}}
\newcommand\C{{\mathbb C}}
\newcommand\F{{\mathbb F}}
\newcommand\G{{\mathbb G}}

\newcommand\N{{\mathbb N}}

\newcommand\R{{\mathbb R}}
\renewcommand\S{{\mathbb S}}
\newcommand\T{{\mathbb T}}
\newcommand\Z{{\mathbb Z}}

\newcommand\bP{{\mathbf P}}
\newcommand\bR{{\mathbf R}}
\newcommand\bS{{\mathbf S}}
\newcommand\bT{{\mathbf T}}
\newcommand\bX{{\mathbf X}}

\newcommand\bk{{\mathbf{k}}}
\newcommand\bv{{\mathbf{v}}}
\newcommand\bw{{\mathbf{w}}}

\newcommand\Fun{{\F_1}}

\renewcommand\int{\textup{int}}

\newcommand\pos{{\textup{pos}}}

\newcommand\alg{\textup{alg}}

\newcommand\an{\textup{an}}

\newcommand\idem{\textup{idem}}

\newcommand\core{\textup{core}}

\newcommand\mon{\textup{mon}}

\newcommand\ev{\textup{ev}}
\newcommand\trop{\textup{trop}}
\renewcommand\log{\textup{log}}

\newcommand\gp{\textup{gp}}

\newcommand\hyp{\textup{hyp}}

\renewcommand\={\equiv}
\renewcommand\geq{\geqslant}
\renewcommand\leq{\leqslant}
\newcommand{\gen}[1]{\langle #1 \rangle}
\newcommand{\bpquot}[2]{#1\!\sslash\!#2}
\newcommand{\bpgenquot}[2]{#1\!\sslash\!\gen{#2}}
\newcommand{\bigbpgenquot}[2]{#1\big\slash\!\!\!\big\slash\!\big\langle\,{#2}\,\big\rangle}

\newcommand{\hyperplus}{{\,\raisebox{-1.1pt}{\larger[-0]{$\boxplus$}}\,}}

\title{Tropical geometry over the tropical hyperfield}

\author{Oliver Lorscheid}
\address{\rm Instituto Nacional de Matem\'atica Pura e Aplicada, Rio de Janeiro, Brazil}
\email{{oliver@impa.br}}
\thanks{The author thanks the Max Planck Institute for Mathematics that hosted and supported him during the preparation of this manuscript.}

\begin{document}

\begin{abstract}
 In this text, we merge ideas around the tropical hyperfield with the theory of ordered blueprints to give a new formulation of tropical scheme theory. The key insight is that a nonarchimedean absolute value can be considered as a morphism into the tropical hyperfield. In turn, ordered blueprints make it possible to consider the base change of a classical variety to the tropical hyperfield. We call this base change the \emph{scheme theoretic tropicalization} of the classical variety.
 
 Our first main result describes the Berkovich analytification and the tropicalization of a classical variety as sets of rational points of scheme theoretic tropicalizations, including a characterization of the respective topologies. Our second main result shows that the Giansiracusa bend relations can be derived by a natural construction from the scheme theoretic tropicalization.
\end{abstract}

\maketitle


\begin{small} \tableofcontents \end{small}


\section*{Introduction}
\label{introduction}


\subsection*{History of the tropical hyperfield}

While hyperrings were defined as early as 1956 by Krasner (\cite{Krasner57}), the tropical hyperfield $\bT$ was introduced more recently in 2011 by Oleg Viro (\cite{Viro11}), with a reformulation of tropical geometry in mind. Around the same time Connes and Consani (\cite{Connes-Consani10}) recognized the relevance of hyperfields for absolute arithmetic and later found back Viro's tropical hyperfield from their point of view (\cite{Connes-Consani15}). A closely related notion are Izhakian's extended tropical numbers (\cite{Izhakian09}), which in fact were introduced before the tropical hyperfield, though the relation between these objects was understood only later (cf.\ Remark \ref{rem: extended tropical numbers}).

In 2013, the seminal paper \cite{Giansiracusa-Giansiracusa16} by Jeffrey and Noah Giansiracusa inaugurated tropical scheme theory, a new branch of tropical geometry that seeks for a scheme theoretic formulation of tropical geometry. Maclagan and Rinc\'on (\cite{Maclagan-Rincon14}) showed soon after that the weights of tropical varieties are encoded in the scheme structure, and the author (\cite{Lorscheid15}) put this theory on a more sophisticated footing using ordered blueprints. This latter paper contains the observation that hyperrings are ordered blueprints, which provides a scheme theory for hyperrings as a byproduct. A variation of algebraic geometry over hyperrings was developed independently by Jun (\cite{Jun17}, \cite{Jun18}) while exploring the relation to tropical geometry from an altered angle.

In 2016, Baker and Bowler (\cite{Baker-Bowler19}) formulated matroid theory with coefficients in a hyperfield. In particular, matroids over the tropical hyperfield turn out to valuated matroids, aka tropical linear spaces following Speyer (\cite{Speyer08}). A joint follow-up work of Baker and the author (\cite{Baker-Lorscheid18}) uses scheme theory for ordered blueprints to construct moduli spaces of matroids. In particular, the moduli space of tropical linear spaces, aka the Dressian, is an object over the tropical hyperfield.


\subsection*{Intention and scope of this text}

Our exposition is meant as a reader friendly introduction to tropical scheme theory from the particular perspective of the tropical hyperfield. While it is build on ideas and theories that were developed in the aforementioned works, this text explores a new mixture of hyperfields with algebraic geometry for ordered blueprints. As a result, we gain a frame work for tropical scheme theory that has certain advantages over previous approaches that are based on the tropical semifield.

Since a nonarchimedean absolute value can be interpreted as a morphism into the tropical hyperfield $\bT$, we can define the scheme theoretic tropicalization of a classical variety \emph{literally} as the base change along this morphism into $\bT$. One of our main efforts in this text is to show that the Kaijiwara-Payne tropicalization emerges from the scheme theoretic tropicalization as the set of $\bT$-rational points, including a characterization of its topology. A surprising insight to us is that the Giansiracusa bend relations appear naturally from the scheme theoretic tropicalization by enforcing the relation $1+1=1$.

All this makes us believe that the tropical hyperfield is a promising tool for tropical scheme theory, and we hope that this text stimulates future developments in this direction.


\subsection*{The tropical hyperfield}

Let us introduce the protagonist of our text, which acts out as a subtle variant of the \emph{tropical semifield} $\overline\R$. We begin with a description of $\overline\R$, which appears in different incarnations in the literature: while the \emph{min-plus-algebra} and the \emph{max-plus-algebra} support the piecewise linear aspect of tropical varieties by using logarithmic coordinates, the \emph{Berkovich model} is a simpler object from an algebraic perspective. It is this latter model of $\overline\R$ that we employ in our text. Its underlying set is $\R_{\geq0}$, its multiplication is the usual multiplication of real numbers and its addition is defined by the rule $a+b=\max\{a,b\}$ where the maximum is taken with respect to the usual linear ordering of the real numbers.

The \emph{tropical hyperfield} $\bT$ has the same underlying set $\R_{\geq0}$ and the same multiplication as $\overline\R$, but the addition of $\overline\R$ gets replaced by the \emph{hyperaddition} that associates with two elements $a,b\in\bT$ the following \emph{subset} $a\hyperplus b$ of $\bT$: 
\[
 a\hyperplus b \ = \ \begin{cases}
                      \big\{\max\{a,b\}\big\} &\text{if }a\neq b, \\
                      [0,a]                   &\text{if }a=b.
                     \end{cases}
\]


\subsection*{Advantages of the tropical hyperfield}
In the following, we will list a number of advantages of the tropical hyperfield $\bT$ over the tropical semifield $\overline\R$, which shall underline the potential of $\bT$ for tropical scheme theory.

\begin{advantage}
 The hyperaddition of $\bT$ is characterized by the property that for a field $k$, the strict triangular inequality $v(a+b)\leq\max\{v(a),v(b)\}$ of a nonarchimedean absolute value $v:k\to \R_{\geq0}$ is equivalent with the condition $v(a+b)\in v(a)\hyperplus v(b)$. This allows us to consider nonarchimedean absolute values as morphisms in a suitable category (Theorem \ref{thm: nonarchimedean seminoms as morphisms}). Thus we can consider the base change of a variety over $k$ to $\bT$.
\end{advantage}

\begin{advantage}
 The hyperaddition of $\bT$ provides us with a notion of additive inverses: for every element $a\in\bT$, there is a unique element $b\in\bT$ such that $0\in a\hyperplus b$, namely $b=a$. More generally, we have $0\in a_1\hyperplus\dotsb\hyperplus a_n$ if and only if the maximum occurs twice among the summands $a_i$ (Lemma \ref{lemma: characterization of monomial relations in the tropical hyperfield}). This allows us to reformulate the corner locus of a tropical polynomial $p$ as the set of points $x$ such that $0\in p(x)$ (Remark \ref{rem: corner locus as zero set using the tropical hyperfield}).
\end{advantage}

\begin{advantage}
 A tropical linear space is the geometric realization of a valuated matroid or a $\bT$-matroid in the language of \cite{Baker-Bowler19}. Therefore it seems natural to use $\bT$ as a basis for tropical geometry. In particular, the moduli space of tropical linear spaces is an object defined over $\bT$, as explained in \cite{Baker-Lorscheid18}.
\end{advantage}

\begin{advantage}
 In contrast to the polynomial ring $\overline\R[T_1,\dotsc,T_n]$, the ambient semiring of the free algebra $\bT[T_1,\dotsc,T_n]$ over $\bT$ is a domain, i.e.\ the multiplication by nonzero elements defines an injective map. This might be helpful to establish a theory of prime ideals and to formulate a criterion for the irreducibility of a tropical scheme. For some details on $\bT[T]$, cf.\ \cite[Appendix A]{Baker-Lorscheid18b}.
\end{advantage}


\subsection*{Ordered blueprints}
While hyperfields and hyperrings are coming short of certain properties that are required for tropical geometry, such as free objects and tensor products, the more ample notion of ordered blueprints has proven to be a suitable tool for tropical scheme theory (cf.\ \cite{Baker-Lorscheid18} and \cite{Lorscheid15}). Therefore we will refrain from spelling out the axiomatic of hyperfields, but we rather consider $\bT$, along with other algebraic objects of interest, as ordered blueprints.

In this text, all semirings are commutative with $0$ and $1$. An \emph{ordered blueprint} is a triple $B=(B^\bullet, B^+,\leq)$ where $B^+$ is a semiring, $B^\bullet$ is a multiplicatively closed set of generators of $B^+$ that contains $0$ and $1$ and $\leq$ is a partial order on $B^+$ that is \emph{additive} and \emph{multiplicative}, i.e.\ $x\leq y$ implies $x+z\leq y+z$ and $xz\leq yz$ for all $x,y,z\in B^+$. We say that $\leq$ is \emph{generated} by a set of relations $\{x_i\leq y_i\}$ if it is the smallest additive and multiplicative partial order that contains the relations $x_i\leq y_i$.

We can realize the tropical hyperfield $\bT$ as the following ordered blueprint $(\bT^\bullet,\bT^+,\leq)$: its \emph{underlying monoid} $\bT^\bullet$ is the set $\R_{\geq0}$ of nonnegative real numbers together with the usual multiplication, its \emph{ambient semiring} $\bT^+$ is the monoid semiring $\N[\R_{>0}]$ of finite formal sums $\sum a_i$ of positive real numbers $a_i$, and its partial order $\leq$ is generated by the relations $c\leq a+b$ for which $c\in a\hyperplus b$. 

Let $B$ and $C$ be ordered blueprints. A morphism $f:B\to C$ is a map $f^\bullet:B^\bullet\to C^\bullet$ with $f(0)=0$, $f(1)=1$ and $f(ab)=f(a)f(b)$ for all $a,b\in B^\bullet$ that extends (necessarily uniquely) to an order-preserving semiring homomorphism $f^+:B^+\to C^+$. This defines the category $\OBlpr$ of ordered blueprints.


\subsection*{Valuations as morphisms}

Let $k$ be a field and $v:k\to\R_{\geq0}$ a nonarchimedean absolute value, i.e.\ $v(0)=0$, $v(1)=1$, $v(ab)=v(a)v(b)$ and $v(c)\leq\max\{v(a),v(b)\}$ whenever $c=a+b$. 

Using the identification $\R_{\geq0}=\bT^\bullet$, the relation $v(c)\leq\max\{v(a),v(b)\}$ is equivalent with $v(c)\leq v(a)+v(b)$ where $\leq$ is the partial order of $\bT$. 

We associate with $k$ the ordered blueprint $\bk=(\bk^\bullet,\bk^+,\leq)$ where $\bk^\bullet$ is the multiplicative monoid of $k$, $\bk^+=\N[k^\times]$ is the group semiring generated by $k^\times$ and $\leq$ is generated by all relations $c\leq a+b$ (considered as elements of $\bk^+$) for which $c=a+b$ in $k$. 

Under these identifications, $v:k\to\R_{\geq0}$ defines a multiplicative map $\bv^\bullet:\bk^\bullet\to \bT^\bullet$ that extends to an order-preserving semiring homomorphism $\bv^+:\bk^+\to\bT^+$. In other words, the nonarchimedean absolute value $v:k\to\R_{\geq0}$ corresponds to a morphism $\bv:\bk\to\bT$ of ordered blueprints (Theorem \ref{thm: nonarchimedean seminoms as morphisms}).


\subsection*{Tropicalization as a base change}

For an ordered blueprint $B$, we denote by $\Spec B$ the corresponding object of the dual category of $\OBlpr$ and call it an \emph{affine ordered blue scheme}. Note that the category $\OBlpr$ contains tensor products $C\otimes_BD$, this is, the colimits of diagrams of the form $C\leftarrow B\to D$.

Given a morphism $\bk\to B$ of ordered blueprints, we define the \emph{scheme theoretic tropicalization of $\Spec B$ along $\bv:\bk\to\bT$} as $\Spec\big(B\otimes_{\bk}\bT\big)$, which can be thought of as the base change of $\Spec B$ from $\bk$ to $\bT$. Note that the tensor product comes with a canonical morphism $\bT\to B\otimes_{\bk}\bT$.


\subsection*{The Kajiwara-Payne tropicalization}

Let $X$ be an affine $k$-scheme with coordinate ring $R$. While the \emph{Berkovich analytification of $X$} is defined intrinsically as the set $X^\an$ of seminorms $w:R\to\R_{\geq0}$ that extend $v$ to $R$, the tropicalization of $X$ requires an additional choice of coordinates, e.g.\ in form of a closed immersion $\iota:X\to T$ into an affine toric variety $T=\Spec k[A]$ over $k$. Pulling back global sections defines a map $\iota^\sharp_A:A\to R$. The \emph{Kajiwara-Payne tropicalization of $X$} is defined as the image $X^\trop=\trop(X^\an)$ of the map
\[
 \begin{array}{cccc}
  \trop: & X^\an            & \longrightarrow & \Hom(A,\R_{\geq0}). \\
         & w:R\to\R_{\geq0} & \longmapsto     & w\circ\iota^\sharp_A
 \end{array}
\]
See Remark \ref{rem: the Kajiwara-Payne tropicalization map is surjective} for the relation with more common definition in terms of the bend locus.


\subsection*{The Kajiwara-Payne tropicalization as a rational point set}

Jun observes in \cite{Jun17} that the Berkovich analytification of $X$ corresponds to the hyperring morphism from $R$, considered \emph{as a hyperring}, into the tropical hyperfield $\bT$. We transfer this approach to ordered blueprints, which allows us to recover both the analytification and the tropicalization of $X=\Spec R$ as $\bT$-rational point sets of the following scheme theoretic tropicalizations. 

We associate with $R$ the following ordered blueprint $\bR$. Its underlying monoid $\bR^\bullet$ is the multiplicative monoid of $R$. Its ambient semiring $\bR^+$ is the monoid semiring $\N[\bR^\bullet]$ modulo the identification of $0\in \bR^\bullet$ with the empty sum. Its partial order is generated by the relation $c\leq a+b$ (considered as elements of $\bR^+$) for which $c=a+b$ in $R$. We define $\bX=\Spec \bR$. 

Let $\iota^\sharp:k[A]\to R$ be the surjection that pulls back global sections along the closed immersion $\iota:X\to T$. Let $B$ be the following ordered blueprint. Its underlying monoid $B^\bullet$ is the submonoid $\{\iota^\sharp(ra)|r\in k,a\in A\}$ of $R$. Its ambient semiring $B^+$ is the monoid semiring $\N[B^\bullet]$ modulo the identification of $0\in B^\bullet$ with the empty sum. Its partial order $\leq$ is generated by the relation $c\leq a+b$ (with $a,b,c\in B^\bullet$) for which $c=a+b$ in $R$. We define $Y=\Spec B$. 

Note that the inclusion $B\to \bR$ induces a $\bT$-linear morphism $f:B\otimes_{\bk}\bT\to \bR\otimes_{\bk}\bT$. We define the sets of $\bT$-linear morphisms
\[
 \Trop_\bv(\bX)(\bT) \ = \ \Hom_\bT(\bR\otimes_{\bk}\bT,\bT) \quad \text{and} \quad \Trop_\bv(Y)(\bT) \ =\ \Hom_\bT(B\otimes_{\bk}\bT,\bT).
\]
Composing with $f$ defines a map $f^\ast:\Trop_\bv(\bX)(\bT) \to \Trop_\bv(Y)(\bT)$.

\begin{thmA}\label{thmA}
 There are natural bijections $X^\an\to\Trop_\bv(X)(\bT)$ and $X^\trop\to\Trop_\bv(Y)(\bT)$ such that the diagram
 \[
  \begin{tikzcd}[column sep=3cm]
   X^\an \ar{r}{\trop} \ar{d} & X^\trop \ar{d} \\
   \Trop_\bv(\bX)(\bT) \ar{r}{f^\ast} & \Trop_\bv(Y)(\bT)
  \end{tikzcd}
 \]
 commutes. 
\end{thmA}

In fact, the bijections in this statement are homeomorphisms with respect to topologies that stem from the Euclidean topology of $\bT=\R_{\geq0}$; cf.\ section \ref{subsection: topology for rational point sets} for details.


\subsection*{The Giansiracusa bend}

Jeff and Noah Giansiracusa introduce in \cite{Giansiracusa-Giansiracusa16} the bend relation for tropical polynomials, which allows them to prove an analogous result to Theorem \ref{thmA} for scheme theoretic tropicalizations over the tropical semifield $\overline\R$. The following result shows that the scheme theoretic tropicalization over $\bT$ recovers the Giansiracusa bend in a natural way.

Recall the context of the Kajiwara-Payne tropicalization: $k$ is a field with nonarchimedean absolute value $v:k\to\R_{\geq0}$ and $X=\Spec R$ is a $k$-scheme together with a closed immersion $\iota:X\to T$ into a toric $k$-variety of the form $T=\Spec k[A]$.

The \emph{Giansiracusa bend $\Bend_{v,\iota}^{GG}(X)$ of $X$ (along $v$ with respect to $\iota$)} is defined as the spectrum of the quotient of the free $\overline\R$-algebra $\overline\R[A]$ of finite $\overline\R$-linear combinations $\sum t_aa$ of elements $a\in A$ by the relations of the form 
\[\textstyle
 v(c_a) a+\sum v(c_j) b_j \ = \ \sum v(c_j) b_j
\]
for $c_a,c_j\in k$ and $a, b_j\in A$ with $c_a a=\sum c_j b_j$ in $R$. 

We consider the \emph{field with one element} as the ordered blueprint $\Fun=\big(\{0,1\},\N,=\big)$ and the \emph{Boolean semifield} as the ordered blueprint $\B=\big(\{0,1\},\{0,1\},=\big)$ where ``$=$'' stand for the trivial partial order and the addition of $\B^+=\{0,1\}$ is characterized by the equation $1+1=1$.

\begin{thmA}\label{thmB}
 There is a canonical isomorphism $\Bend_{v,\iota}^{GG}(R) \to (\Trop_\bv(B)\otimes_\Fun\B)^+$.
\end{thmA}

In fact, we prove a stronger version of Theorem \ref{thmB}, in which we express the refinement of $\Bend_{v,\iota}(R)$ as a blueprint in terms of $\Trop_\bv(B)$. While this latter result can be found as Theorem \ref{thm: recovering the bend relation} in this text, Theorem \ref{thmB} appears as Corollary \ref{cor: the Giansiracusa bend from the scheme theoretic tropicalization}.


\subsection*{The guiding example}

We illustrate the main concepts and results of this paper at the appropriate positions in the case of the standard plane line defined by the polynomial $T_1+T_2+1$. We summarize these explanations in the following in order to exemplify Theorems \ref{thmA} and \ref{thmB}.

Let $k$ be a field with nonarchimedean absolute value $v:k\to\R_{\geq0}$. Let $X$ be the closed subscheme of the affine plane $\A^2_k$ over $k$ that is defined by the polynomial $T_1+T_2+1$, which comes with a closed immersion $\iota:X\to\A^2_k$ and coordinate ring $R=k[T_1,T_2]/(T_1+T_2+1)$. 

\subsubsection*{The set theoretic tropicalization}
The tropicalization of $X$ is 
\[
 X^\trop \ = \ \big\{ \, (a_1,a_2)\in \R_{\geq0}^2 \,\big| \, \text{the maximum among $a_1$, $a_2$ and $1$ occurs twice} \, \big\},
\]
which is also called the \emph{bend locus of $T_1+T_2+1$}; cf.\ section \ref{subsection: the Kajiwara-Payne tropicalization} for details and Example \ref{ex: topologies of tropical plane and tropical plane line} for an illustration.

\subsubsection*{The scheme theoretic tropicalization}
We turn to a description of the scheme theoretic tropicalization of $X$ with respect to its embedding into the affine plane over $k$. Let $\bv:\bk\to\bT$ be the morphism associated with $v:k\to\R_{\geq0}$, cf.\ Theorem \ref{thm: nonarchimedean seminoms as morphisms}. The associated ordered blueprint $B$ is as follows: its ambient semiring is the polynomial algebra
\[
 B^+ \ = \ \big(\N[k^\times]\big)\, [T_1,T_2]
\]
where $\N[k^\times]$ is the group semiring of finite formal sums of elements of $k^\times$. Its underlying monoid consists of all terms of the form $cT^{e_1}T^{e_2}$ where $c\in k$ and $e_1,e_2\in\N$. Its partial order is generated by the relations
\[
 0 \ \leq \ T_1+T_2+1, \quad -T_1 \ \leq \ T_2+1, \quad -T_2 \ \leq \ T_1+1 \quad \text{and} \quad -1 \ \leq \ T_1+T_2.
\]
By Lemma \ref{lemma: the base change to the tropical hyperfield as quotient of a free algebra}, the tropicalization $\Trop_\bv(B)=B\otimes_\bk\bT$ of $B$ has the following explicit description: the association $cT_1^{e_1}T_2^{e_2}\otimes t\mapsto v(c)tT_1^{e_1}T_2^{e_2}$ defines an isomorphism
\[
 \Trop_\bv(B)^+ \ \simeq \ \big(\N[\R_{>0}] \big)\, [T_1,T_2],
\]
of semirings that identifies the underlying monoid of $\Trop_\bv(B)$ with the submonoid of $\N[\R_{>0}][T_1,T_2]$ that consists of all terms of the form $tT_1^{e_1}T_2^{e_2}$ with $t\in\R_{\geq0}$ and $e_1,e_2\in\N$. The partial order of $\Trop_\bv(B)$ coincides with the partial order of $\N[\R_{>0}][T_1,T_2]$ that is generated by the defining relations of the partial order of $\bT^+=\R_{\geq0}\subset\N[\R_{>0}]$ together with the relations
\[
 0 \ \leq \ T_1+T_2+1, \quad  T_1 \ \leq \ T_2+1, \quad  T_2 \ \leq \ T_1+1\quad \text{and}\quad  1 \ \leq \ T_1+T_2,
\]
cf.\ Example \ref{ex: the Kajiwara-Payne tropicalization from the scheme theoretic tropicalization} for details.

\subsubsection*{Recovering the set theoretic tropicalization}
Theorem \ref{thmA} asserts that the tropicalization $X^\trop$ equals the set of $\bT$-linear homomorphisms $f:\Trop_\bv(B)\to\bT$. Mapping $f$ to $\big(f(T_1),f(T_2)\big)$ defines a bijection of $\Hom_\bT(\Trop_\bv(B),\bT)$ with 
\[
 \big\{\, (a_1,a_2)\in\bT^2 \, \big| \, 0\leq a_1+a_2+1,\, a_1 \leq a_2+1,\, a_2 \leq a_1+1,\, 1 \leq a_1+a_2 \, \big\}.
\]
By Lemma \ref{lemma: characterization of monomial relations in the tropical hyperfield}, each of the four defining relations on $(a_1,a_2)$ is equivalent with the condition that the maximum among $a_1$, $a_2$ and $1$ occurs twice. Thus this latter set is precisely $X^\trop$, as claimed in Theorem \ref{thmA}.

\subsubsection*{The Giansiracusa bend}
We turn to Theorem \ref{thmB}, which exhibits the bend relations in terms of the scheme theoretic tropicalization. In our example, the Giansiracusa bend is the semiring
\[
 \Bend_{v,\iota}^{GG}(R) \ = \ \overline\R[T_1,T_2] \big/ \big\langle T_1+T_2+1\sim T_1+T_2\sim T_1+1\sim T_2+1 \big\rangle,
\]
cf.\ Example \ref{ex: the Giansiracusa bend}. On the other hand, we have the identifications
\[
 \big(\, \Trop_\bv(B)\otimes_\Fun\B\, \big)^+ \ = \ \Trop_\bv(B)^+\otimes_\N\B^+ \ = \ \Trop_\bv(B)^+ \, / \, \langle 1+1 \sim 1 \rangle,
\]
which show that $x+x=x$ holds for every $x$ in $(\Trop_\bv(B)\otimes_\Fun\B)^+$, cf.\ section \ref{subsection: idempotent ordered blueprints}. Under the identification $\Trop_\bv(B)^+\simeq\N[\R_{>0}][T_1,T_2]$, we obtain that 
\[
 T_1+T_2+1 \ \leq \ T_1 + T_2 + T_1+T_2 \ = \ T_1+T_2
\]
(using $1\leq T_1+T_2$) and that
\[
 T_1+T_2 \ = \ T_1+T_2+0 \ \leq \ T_1+T_2+1+1 \ = \ T_1+T_2+1
\]
(using $0\leq 1+1$) holds in $(\Trop_\bv(B)\otimes_\Fun\B)^+$. Thus we gain the equality $T_1+T_2+1=T_1+T_2$ in $(\Trop_\bv(B)^\idem)^+$. Repeating the same argument with the roles of $T_1$, $T_2$ and $1$ exchanged yields the defining relations
 \[
  T_1+T_2+1 \ = \ T_1+T_2 \ = \ T_1+1 \ = \ T_2+1
 \]
 of $\Bend_{v,\iota}^{GG}(R)$. This illustrates Theorem \ref{thmB} in our example.


\subsection*{Remark on variations}

Let $\bk$ be as before. Whenever we have a morphism $v:\bk\to C$ into some ordered blueprint $C$, we can consider the base change $\Spec\big(B\times_\bk C\big)$ of an affine ordered blue $\bk$-scheme $Y=\Spec B$ to $C$ along this morphism, which should be thought of as the scheme theoretic tropicalization of $Y$ over $C$.

Tropicalizations along the following morphisms $v:\bk\to C$ might produce interesting theories. First of all, we can consider higher rank valuations $v:k\to \R_{\geq0}^n$ (where we use the exponential notation) as a morphism $v:\bk\to\bT^{(n)}$ where $\bT^{(n)}$ is the ordered blueprint with underlying monoid $A=\R_{>0}^n\cup\{0\}$, with ambient semiring $\N[\R^n_{>0}]$ and with the partial order that is generated by relations of the form
\[
 (c_1,\dotsc,c_n) \ \leq \ (a_1,\dotsc,a_n) + (b_1,\dotsc,b_n)
\]
for which there is an $i\in\{1,\dotsc,n\}$ such that $a_j=b_j=c_j$ for $j<i$ and such that the maximum among $a_i$, $b_i$ and $c_i$ appears twice.

Another interesting example is the sign map $\R\to\{0,\pm1\}$, which can be interpreted as a morphism $\sign:\bk\to\bS$ where $\bk$ is the ordered blueprint associated with $k=\R$ and where $\bS$ is the \emph{sign hyperfield}, which has the following shape as an ordered blueprint. Its underlying monoid is $\bS^\bullet=\{0,\pm1\}$, its ambient semiring is $\bS^+=\N[1,-1]$ (where $-1$ has to be understood as a symbol and not as an additive inverse of $1$) and its partial order is generated by the relations
\[
 0 \ \leq \ 1+(-1) \qquad 1 \ \leq \ 1+(-1) \qquad \text{and} \qquad 1\ \leq \ 1+1.
\]
Tropicalizations along $\sign:\bk\to\bS$ might be useful to study real algebraic varieties. In particular, this might bring new insights to questions around the (disproven) Macphersonian conjecture; cf.\ \cite{Mnev-Ziegler93} and \cite{Liu17}.

A variation of the sign map is the phase map $\C\to\S^1\cup\{0\}$, which assigns to a nonzero complex number $z$ its argument $z/|z|$ on the unit circle $\S^1$. This map can be realized as a morphism $\bk\to\bP$ where $\bk$ is associated with $k=\C$ and where $\bP$ is the \emph{phase hyperfield}. We omit a description of $\bP$, but refer to section 2.1 in \cite{Anderson-Davis19} for details and further variations.


\subsection*{Divergence in notation}

In this text, we aim for a simplified account of scheme theoretic tropicalization, in contrast to the broader context of \cite{Lorscheid15}. The specific situation of this paper---namely, the restriction to affine ordered blue schemes and the fixed tropicalization base $\bT$---allows us to simplify the exposition of our results considerably. In the following, we point out the major differences to \cite{Baker-Lorscheid18}, \cite{Lorscheid15} and \cite{Lorscheid18} in order to avoid confusion when comparing these writings.

Most notably, we are using different incarnations of the tropical numbers in this text, whose notations deviate from that used in \cite{Baker-Lorscheid18}, \cite{Lorscheid15} and \cite{Lorscheid18}. Namely, in \cite{Lorscheid15} and \cite{Lorscheid18} the symbol $\T$ is used for both the tropical semifield, which is denoted by $\overline\R$ in this text, and the associated ordered blueprint $\T=(\overline\R,\overline\R,=)$, which agrees with the notation in this text. In \cite{Baker-Lorscheid18}, we use $\T$ for the tropical hyperfield, which is denoted by $\bT$ in this text and by $\T^\hyp$  in \cite{Lorscheid18}.

The reason for us to use in \cite{Lorscheid15} and \cite{Lorscheid18} the symbol $\T$ for both the tropical semifield and the associated blueprint is that we identify a semiring $R$ with the associated blueprint $B=(R,R,=)$. Since for this text a different realization of semiring as ordered blueprints stays in the foreground, we make a clear distinction between these different objects.

This distinction has the advantage that there is no ambiguity between the free semiring $R[T]$ over a semiring $R$, whose elements are polynomials, and the free ordered blueprint $B[T]$ over the associated ordered blueprint $B=(R,R,=)$, whose elements are monomials with coefficients in $B$.


\subsection*{Content overview}
In section \ref{section: ordered blueprints}, we review the definition of and some basic facts for ordered blueprints. In section \ref{section: tropicalization as a base change to the tropical hyperfield}, we introduce scheme theoretic tropicalizations, which includes the interpretation of absolute values as morphisms of ordered blueprints, as well as some results on the natural topology for $\bT$-rational point sets. In section \ref{section: recovering the Kajiwara-Payne tropicalization}, we exhibit the Kajiwara-Payne tropicalization as the $\bT$-rational point set of a scheme theoretic tropicalization. In section \ref{section: the relation between the tropical hyperfield and the tropical semifield}, we explain the relation between the tropical hyperfield and the tropical semifield. In section \ref{section: recovering the bend}, we recover the Giansiracusa bend from the scheme theoretic tropicalization.


\subsection*{Acknowledgements}

The author thanks Matt Baker and Sam Payne for their help with preparing this text.


\section{Ordered blueprints}
\label{section: ordered blueprints}

In this section, we introduce ordered blueprints and some basic constructions such as free algebras, quotients and tensor products. We will explain how we can consider monoids and semirings as ordered blueprints and finally introduce the tropical hyperfield in its incarnation as an ordered blueprint. For more details on ordered blueprints we refer to \cite{Baker-Lorscheid18}, \cite{Lorscheid15} and \cite{Lorscheid18}.


\subsection{Basic definitions}
\label{subsection: basic definitions}

In this text, a \emph{semiring} is always commutative and with $0$ and $1$, i.e.\ both $(R,+,0)$ and $(R,\cdot,1)$ are commutative monoids, multiplication distributes over addition and $0\cdot a=0$ for all $a\in R$.

An \emph{ordered blueprint} is a triple $B=(B^\bullet,B^+,\leq)$ where $B^+$ is a semiring, $B^\bullet$ is a subset of $B^+$ and $\leq$ is a partial order on $B^+$ such that
\begin{enumerate}
 \item $B^\bullet$ is closed under multiplication, contains $0$ and $1$ and generates $B^+$ as a semiring;
 \item $\leq$ is \emph{additive} and \emph{multiplicative}, i.e.\ $x\leq y$ implies $x+z\leq y+z$ and $xz\leq yz$ for all $x,y,z\in B^+$. 
\end{enumerate}
We call $B^\bullet$ the \emph{underlying monoid}, $B^+$ the \emph{ambient semiring} and $\leq$ the \emph{partial order} of the ordered blueprint $B$. 

We typically denote the elements of $B^\bullet$ by $a$, with $b$, $c$ and $d$, and the elements of $B^+$ by $x$, $y$, $z$ and $t$ or $\sum a_i$, $\sum b_j$, $\sum c_k$ and $\sum d_l$ where we assume that the $a_i$, $b_j$, $c_k$ and $d_l$ are elements of $B^\bullet$. Note that every element of $B^+$ is indeed a sum of elements in $B^\bullet$. 

We consider $B^\bullet$ as the underlying set of the ordered blueprint $B$, and we say that $a$ is an element of $B$ if $a\in B^\bullet$.

A \emph{morphism of ordered blueprints $f:B_1\to B_2$} is a multiplicative map $f^\bullet:B_1^\bullet \to B_2^\bullet$ with $f(0)=0$ and $f(1)=1$ that extends to an order preserving semiring homomorphism $f^+:B_1^+\to B_2^+$. Note that $f^+$ is uniquely determined by $f^\bullet$ since $B_1^\bullet$ generates $B_1^+$ as a semiring. This defines the category $\OBlpr$ of ordered blueprints. 

In the following, we will introduce several constructions and subclasses of ordered blueprints as well as several explicit examples of ordered blueprints.


\subsection{Free algebras}
\label{subsection: free algebras}

Let $k$ be an ordered blueprint. An \emph{ordered blue $k$-algebra} is an ordered blueprint $B$ together with a morphism $k\to B$, which we call the \emph{structure map of $B$}. We often refer to an ordered blue $k$-algebra by $B$ without mentioning the structure map explicitly. A $k$-linear morphism between two ordered blue $k$-algebras $B$ and $C$ is a morphism $f:B\to C$ of ordered blueprints that commutes with the structure maps of $B$ and $C$.

Let $B=(B^\bullet,B^+,\leq)$ be an ordered blueprint and $A$ be a commutative and multiplicatively written monoid. We define the \emph{free ordered blue $B$-algebra in $A$} as the following ordered blue $B$-algebra $B[A]$. Its ambient semiring is the semiring
\[\textstyle
 B[A]^+ \ = \ \big\{ \, \sum\limits_{a\in A} x_aa \, \big| \, x_a\in B^+\text{ and }x_a=0\text{ for almost all }a \, \big\}
\]
whose addition is defined componentwise and its multiplication extends the multiplication of $A$ linearly. We consider $A$ as a submonoid of $B[A]^+$ and we write $a$ for the element $\sum c_bb$ with $c_a=1$ and $c_b=0$ for $b\neq a$. In particular, we consider $1$ as an element of $A\subset \N[A]^+$.

The underlying monoid of $B[A]$ is defined as the subset
\[\textstyle
 B[A]^\bullet \ = \ \big\{ \, ca \in B[A]^+\, \big| \, c\in B^\bullet, a\in A \, \big\}
\]
and its partial order is generated by the relations of the form $x1\leq y1$ with $x,y\in B^+$ whenever $x\leq y$ in $B^+$.

The association $c\mapsto c1$ defines a morphism of ordered blueprints $B\to B[A]$, which endows $B[A]$ with the structure of an ordered blue $B$-algebra. It satisfies the universal property that every monoid morphism $A\to C^\bullet$ into the underlying monoid of an ordered blue $B$-algebra $C$ extends uniquely to a morphism $B[A]\to C$ of ordered blue $B$-algebras.

\begin{ex}
 In the case that $A=\{T_1^{e_1}\dotsb T_n^{e_n}|e_i\in\N\}$ consists of the monomials in $T_1,\dotsc, T_n$, we write $B[T_1,\dotsc,T_n]=B[A]$, which we think of as the ``polynomial algebra'' over $B$. Its ambient semiring is the usual polynomial semiring $B^+[T_1,\dotsc,T_n]$ and its underlying monoid consists of the monomials $cT_1^{e_1}\dotsb T_n^{e_n}$ with coefficient $c\in B^\bullet$.
 
 Another example of importance for tropical geometry are Laurent polynomial algebras, which are the free algebras associated with the group $A=\{T_1^{e_1}\dotsb T_n^{e_n}|e_i\in\Z\}$ of Laurent monomials. In this case, we write $B[T_1^{\pm1},\dotsc,T_n^{\pm1}]=B[A]$. Its ambient semiring is the usual semiring $B^+[T_1^{\pm1},\dotsc,T_n^{\pm1}]$ of Laurent polynomials and its underlying monoid consists of the Laurent monomials $cT_1^{e_1}\dotsb T_n^{e_n}$ with coefficient $c\in B^\bullet$.
\end{ex}


\subsection{Algebraic blueprints}
\label{subsection: algebraic blueprints}

Let $R$ be a semiring. The \emph{trivial partial order on $R$} is the partial order $\leq$ with $x\leq y$ only if $x=y$ for all $x,y\in R$. We will refer to the trivial partial order by $=$. An \emph{algebraic blueprint}, or simply \emph{blueprint}, is an ordered blueprint $B$ whose partial order is trivial. 

Let $B=(B^\bullet,B^+,\leq)$ be an ordered blueprint. Then its \emph{algebraic core} is the blueprint $B^\core=(B^\bullet,B^+,=)$ where we replace the partial order $\leq$ of $B$ by the trivial partial order $=$. The identity map on $B^\bullet$ defines a morphism $B^\core\to B$ of ordered blueprints.

\begin{ex}\label{ex: Boolean semifield}
 The Boolean semifield is the semiring $R=\{0,1\}$ whose multiplication is determined by the axioms for $0$ and $1$ and whose addition is determined by the rule $1+1=1$. We identify $R$ with the algebraic blueprint $\B=(\{0,1\},R,=)$ and call $\B$ by abuse of language the \emph{Boolean semifield}.
\end{ex}


\subsection{Monoids with zero}
\label{subsection: monoids with zero}

A \emph{monoid with zero} is a multiplicatively written commutative monoid $A$ with a distinguished element $0$, called the \emph{zero of $A$}, that satisfies $0\cdot a=0$ for all $a\in A$. We define $A^+=\N[A]^+/\sim$ as the quotient of $\N[A]^+$ by the relation that identifies the zero $0$ of $A\subset\N[A]^+$ with the additively neutral element $\sum 0a$ of $\N[A]^+$.

Note that every element of $A^+$ can uniquely written as a sum $\sum a_i$ of nonzero elements $a_i\in A$. In particular, the map $A\to \N[A]^+\to A^+$ embeds $A$ as a submonoid of $A^+$. Therefore $A^\alg=(A,A^+,=)$ is a blueprint, which we call the \emph{blueprint associated with $A$}.

This allows us to associate the algebraic blueprint $(A,A^+,=)$ with a monoid with zero $A$. 

Note further that the semiring $A^+$ satisfies the universal property that every multiplicative map $f:A\to R$ into a semiring $R$ with $f(0)=0$ and $f(1)=1$ extends uniquely to a semiring morphism $A^+\to R$. 

\begin{ex}\label{ex: the field with one element}
 The \emph{field with one element} is the ordered blueprint $\Fun=(\{0,1\},\N,=)$, which is associated with the monoid $\{0,1\}$. It is an initial object in $\OBlpr$, which means that there is a unique morphism $\Fun\to B$ into any other ordered blueprint $B$.
\end{ex}


\subsection{Semirings}
\label{subsection: semirings}

Let $R$ be a semiring. We denote the \emph{multiplicative monoid of $R$} by $R^\bullet$. The \emph{associated monomial ordered blueprint} is the ordered blueprint $R^\mon=(R^\bullet,(R^\bullet)^+,\leq)$ where $\leq$ is generated by the \emph{(left) monomial relations} $c\leq a+b$ for which $c=a+b$ in $R$. Note that the underlying set of $R^\mon$ is $R$ itself.

This association is functorial in the sense that a homomorphism $f:R_1\to R_2$ of semirings is tautologically a morphism between the associated monomial ordered blueprints, which we denote by $f^\mon:R_1^\mon\to R_2^\mon$. This embeds the category $\SRings$ of semirings as a full subcategory into $\OBlpr$.

We sometimes denote the associated monomial ordered blueprint by a boldface letter, e.g.\ by $\bR=R^\mon$ and $\bk=k^\mon$ where $k$ is typically a field.

Note that this construction differs from the realization of a semiring $R$ as the algebraic blueprint $(R,R,=)$ from \cite[section 2.10]{Lorscheid15}; we will encounter this latter construction, applied to $\overline\R$, in section \ref{subsection: the tropical semifield as an algebraic blueprint} of this text.

\begin{rem}\label{rem: hyperrings as ordered blueprints}
 At first sight, the definition of $R^\mon$ might seem unmotivated. To give some intuition, we explain its consistency with the association of rings with ordered blueprints passing through hyperrings. Namely, given a ring $R$, one defines a hyperaddition by the rule $a\hyperplus b=\{a+b\}$, which turns $R$ into a hyperring. 
 
 Given a hyperring $R$, the relation $c\in a\hyperplus b$ is not symmetric, but \emph{monomial} in the argument $c$ on the left hand side. This leads to the realization of the hyperring $R$ as the ordered blueprint $(R^\bullet, (R^\bullet)^+,\leq)$ where $R^\bullet$ is the multiplicative monoid of $R$ and $\leq$ is generated by the monomial relations $c\leq a+b$ for which $c\in a\hyperplus b$ in $R$. Also cf.\ Remark 2.8 in \cite{Lorscheid15}.
\end{rem}


\subsection{Quotients by relations}
\label{subsection: quotients by relations}

Given an ordered blueprint $B=(B^\bullet,B^+,\leq_B)$ and a set of relations $S=\{x_i\leq y_i\}_{i\in I}$ with $x_i,y_i\in B^+$, we define the ordered blueprint $C=\bpgenquot BS$ as the following triple $(C^\bullet,C^+,\leq_C)$. Let $\leq'$ be the smallest preorder on $B^+$ that contains $\leq_B$ and $S$ and that is closed under multiplication and addition. We write $x\=y$ if $x\leq y$ and $y\leq x$. Then $\=$ is an equivalence relation on $B^+$, and we define $C^+$ as $B^+/\=$, which inherits naturally the structure of an ordered blueprint since $\leq'$ is closed under multiplication and addition. The preorder $\leq'$ induces a partial order $\leq_C$ on $C^+$, which turns $C^+$ into an ordered semiring. The multiplicative subset $C^\bullet$ is defined as the image of $B^\bullet$ under the quotient map $B^+\to C^+$. 

The quotient $C=\bpgenquot BS$ comes with a canonical morphism $\pi:B\to C$ that satisfies the universal property that for every morphism $f:B\to D$ such that $f(x_i)\leq f(y_i)$ holds in $D$ for every $x_i\leq y_i$ in $S$, there is a unique morphism $\bar f:\bpgenquot BS\to D$ such that $f=\bar f\circ\pi$; cf.\ \cite[Prop.\ 5.3.2]{Lorscheid18} for a proof.

\begin{ex}\label{ex: quotients of ordered blueprints}
 It is immediate that we have $\B=\bpgenquot\Fun{1+1\=1}$ (cf.\ Example \ref{ex: Boolean semifield}) where $1+1\=1$ stands for $1+1\leq 1$ and $1\leq 1+1$. 
 
 Another example is $\F_1^\pos=\bpgenquot\Fun{0\leq1}$, which we will encounter again in section \ref{subsection: totally positive ordered blueprints}. Its underlying monoid is $\{0,1\}$ and its ambient semiring is $\N$. In order to determine the partial order of $\F_1^\pos$ consider two natural numbers $x$ and $y$ and assume that $y$ is larger than $x$, i.e.\ $y=x+z$ for some $z\in\N$. Then the relation $0\leq 1$ implies that $0\leq z$ (multiply $0\leq 1$ by $z$) and $x=x+0\leq x+z=y$ (add $x$ to $0\leq z$). We conclude that the partial order $\leq$ of $\F_1^\pos$ is the natural linear order of $\N$.
\end{ex}


\subsection{Tensor products}
\label{subsection: tensor products}

The category of ordered blueprints is complete and cocomplete. In particular, the tensor product $C\otimes_BD$ of three ordered blueprints $B$, $C$ and $D$ with respect to morphisms $B\to C$ and $B\to D$ exists and satisfies the universal property of a pushout of the diagram $C\leftarrow B\to D$. The tensor product can be constructed as follows.

The semiring $(C\otimes_BD)^+$ is the usual tensor product $C^+\otimes_{B^+}D^+$ of commutative semirings, whose elements are classes of finite sums $\sum c_i\otimes d_i$ of \emph{pure tensors $c_i\otimes d_i$} with respect to the usual identifications. The monoid $(C\otimes_BD)^\bullet$ is defined as the subset of all pure tensors $c\otimes d$ of $(C\otimes_BD)^+$ for which $c\in C^\bullet$ and $d\in D^\bullet$. The partial order on $(C\otimes_BD)^+$ is defined as the smallest partial order that is closed under addition and multiplication and that contains all relations of the forms
\[ \textstyle
 \sum a_i \otimes 1 \leq \sum c_k \otimes 1 \qquad \text{and} \qquad \sum 1\otimes b_j \leq \sum 1\otimes d_l
\]
for which $\sum a_i\leq\sum c_k$ in $C$ and $\sum b_j\leq \sum d_l$ in $D$, respectively. 

\begin{ex}\label{ex: tensor products of ordered blueprints} 
 The tensor product satisfies the usual compatibilities with free algebras and quotients. For example, we have 
 \[
  B\otimes_\Fun\Big(\Fun[T_1,\dotsc,T_n]\Big) \ = \ B[T_1,\dotsc,T_n] \quad \text{and} \quad B\otimes_\Fun\Big(\bpgenquot{\Fun}{0\leq1}\Big) \ = \ \bpgenquot{B}{0\leq 1}.
 \]
\end{ex}


\subsection{The tropical hyperfield}
\label{subsection: the tropical hyperfield}

In this section, we shall introduce Viro's tropical hyperfield in its incarnation as an ordered blueprint. As a hyperfield, it is defined as $\R_{\geq0}$ together with the usual multiplication and the hyperaddition 
\[
 a\hyperplus b \ = \ \begin{cases}
                      \big\{\max\{a,b\}\big\} &\text{if }a\neq b, \\
                      [0,a]                   &\text{if }a=b,
                     \end{cases}
\]
where the maximum is taken with respect to the natural linear order of $\R_{\geq0}$. In other words, $c\in a\hyperplus b$ if and only if the maximum among $a$, $b$ and $c$ occurs twice.

Let $\R_{\geq0}^\bullet$ be the multiplicative monoid of $\R_{\geq0}$ and $(\R_{\geq0}^\bullet)^\alg=\big(\R_{\geq0}^\bullet,(\R_{\geq0}^\bullet)^+,=\big)$ the associated algebraic blueprint, cf.\ section \ref{subsection: monoids with zero}. Intuitively, the ordered blueprint version of the tropical hyperfield results from a symbolic exchange of the relation $c\in a\hyperplus b$ by $c\leq a+b$. Looking more closely at the ordered blueprint 
\[\textstyle
 \bT \ = \ \bpgenquot{(\R_{\geq0}^\bullet)^\alg}{c\leq a+b|c\in a\hyperplus b}
\]
reveals that its underlying monoid $\bT^\bullet$ is $\R_{\geq0}^\bullet$, its ambient semiring is $\bT^+=(\R_{\geq0}^\bullet)^+=\N[\R_{>0}^\bullet]$ and its partial order is defined by all relations $c\leq a+b$ for which the maximum occurs twice among $a$, $b$ and $c$. This latter fact generalizes to multiple term sums as follows. 

\begin{lemma}\label{lemma: characterization of monomial relations in the tropical hyperfield}
 For $n\geq1$ and $a,b_1,\dotsc,b_n\in \bT$, we have $a\leq \sum_{j=1}^n b_j$ in $\bT$ if and only if the maximum occurs twice among the elements $a,b_1,\dotsc,b_n$.
\end{lemma}

\begin{proof}
 We proceed by induction on $n$. If $n=1$, then obviously $a\leq b_1$ if and only if $a=b_1$. If $n=2$, then the claim follows from the definition of $\leq$.
 
 Let $n>2$ and $a\leq\sum_{j=1}^n b_j$. Since $\leq$ is generated by monomial relations with only two terms on the right hand side, there must be a partition of $\{1,\dotsc,n\}$ into smaller nonempty subsets $J_i$ with $i\in I$ and $\# I<n$, a relation $a\leq \sum a_i$ and a relation $a_i\leq \sum_{j\in J_i} b_j$ for every $i\in I$. By the inductive hypothesis, the maximum occurs twice among $a$ and the $a_i$ (with varying $i\in I$) and for every $i\in I$ among $a_i$ and the $b_j$ (with varying $j\in J_i$).
 
 Thus there is some $i\in I$ and $j\in J_i$ such that $b_j$ is the maximum of $a$ and the $b_j$ (for $j\in\{1,\dotsc,n\}$). If we have $b_j=b_{j'}$ for some $j'\in J_i$ different from $j$, then maximum among $a,b_1,\dotsc,b_n$ occurs twice. 
 
 If not, then $a_i=b_j$. If $a=a_i$, then maximum among $a,b_1,\dotsc,b_n$ occurs twice. If not, then $a_i=a_{i'}$ for some $i'\in I$ different from $i$. Then there is a $j'\in J_{i'}$ such that $b_{j'}=a_{i'}=a_i=b_j$. Thus also in the last case, the maximum among $a,b_1,\dotsc,b_n$ occurs twice, which verifies one implication of the claim of the lemma.
 
 Conversely, assume that the maximum occurs twice among $a,b_1,\dotsc,b_n$. First we consider the case that it occurs twice among $b_1,\dotsc,b_n$. After relabeling the elements, we can assume that $b_1=b_2\geq b_i$ for all $i=3,\dotsc,n$. By the inductive hypothesis, we have thus $a\leq \sum_{i=1}^{n-1} b_i$ and $b_n\leq \sum_{i=1}^{n-1} b_i$. If $a\geq b_n$, then the former relation implies that
 \[\textstyle
  a \ \leq \ a+b_n \ \leq \ \sum_{i=1}^n b_i
 \]
 and if $a\leq b_n$, then the latter relation implies that 
 \[\textstyle
  a \ \leq \ b_n+b_n \ \leq \ \sum_{i=1}^n b_i.
 \]
 Thus our claim follows if the maximum occurs twice among $b_1,\dots,b_n$. If $a$ is equal to the maximum of $b_1,\dotsc,b_n$, say $a=b_1$, then $a\leq \sum_{i=1}^{n-1}b_i$ and $a\leq a+b_n$ by the inductive hypothesis. Adding $b_n$ to the former relation yields
 \[\textstyle
  a \ \leq \ a+b_n \ \leq \sum_{i=1}^n b_i,
 \]
 as desired. This concludes the proof of the lemma.
\end{proof}

\begin{rem}\label{rem: extended tropical numbers}
 Izhakian's \emph{extended tropical semiring} (\cite{Izhakian09}) is closely related to the tropical hyperfield $\bT$, as was explained to me by Stephane Gaubert. Namely, the multiplication and the hyperaddition of $\bT$ extends to the family of all singletons $\{a\}$ and all intervals $a^\nu=[0,a]$ (the \emph{ghost elements}) by the rules
 \[
  A\cdot B \ = \ \big\{ \, ab \, \big| \, a\in A, b\in B \, \big\} \qquad \text{and} \qquad A\hyperplus B \ = \ \bigcup_{a\in A,b\in B} a\hyperplus b,
 \]
 which defines a semiring structure on $\big\{\{a\},a^\nu\big|a\in\bT\big\}$. This semiring is isomorphic to Izhakian's extended tropical semiring.
\end{rem}


\section{Tropicalization as a base change to the tropical hyperfield}
\label{section: tropicalization as a base change to the tropical hyperfield}

In this section, we explain the variant of scheme theoretic tropicalization over the tropical hyperfield. Roughly speaking, we interpret a nonarchimedean absolute value $v:k\to\R_{\geq0}$ as a morphism into the tropical hyperfield $\bT$ and define the scheme theoretic tropicalization of a $k$-variety as its base change to $\bT$ along this morphism. 

As we will explain in section \ref{subsection: analytification and tropicalization as rational point sets}, the set theoretic tropicalization can be recovered as the set of $\bT$-rational points from the scheme theoretic tropicalization. In last part of this section, we explain how $\Trop_\bv(X)(\bT)$ inherits a topology from $\bT$.


\subsection{Nonarchimedean seminorms as morphisms}
\label{subsection: nonarchimedean seminorms as morphisms}

Let $R$ be a ring. A nonarchimedean seminorm on $R$ is a map $v:R\to\R_{\geq0}$ that satisfies for all $a,b\in R$ that 
\begin{enumerate}
 \item\label{norm1} $v(0)=0$ and $v(1)=1$;
 \item\label{norm2} $v(ab)=v(a)v(b)$;
 \item\label{norm3} $v(a+b)\leq\max\{v(a),v(b)\}$
\end{enumerate}
where the relation $\leq$ in \eqref{norm3} is the natural linear order of $\R_{\geq0}$. Note that if $R=k$ is a field, then a nonarchimedean seminorm $v:k\to\R_{\geq0}$ is the same as a nonarchimedean absolute value. Though the following fact is well-known, we include a proof for completeness.
 
\begin{lemma}\label{lemma: strong version of the strict triangular inequality}
 Let $R$ be a ring and $v:R\to\R_{\geq0}$ a nonarchimedean seminorm. Let $a=\sum_{j=1}^n b_j$ in $R$. Then the maximum occurs twice among $v(a),v(b_1),\dotsc,v(b_n)$.
\end{lemma}

\begin{proof}
 A simple induction shows that $v(a)\leq \max\{v(b_1),\dotsc,v(b_n)\}$. Thus the maximum among $v(a),v(b_1),\dotsc,v(b_n)$ is attend by $v(b_j)$ for some $j$. Without loss of generality, we can assume that $j=1$. Since $\epsilon^2=1$ implies $\epsilon=1$ in $\R_{\geq0}$, we have $v(-1)=1$ and $v(-a)=v(-1)v(a)=v(a)$. Axiom \eqref{norm3} of a nonarchimedean seminorm applied to $-b_1=-a+\sum_{i=2}^n$ yields
 \[
  v(b_1) \ = \ v(-b_1) \ \leq \ \max\{v(-a),v(b_2),\dotsc,v(b_n)\} \ \leq \ \max\{v(a),v(b_2),\dotsc,v(b_n)\},
 \]
 which implies that the maximum $v(b_1)$ occurs twice among $v(a),v(b_1),\dotsc,v(b_n)$.
\end{proof}
 
Recall that the underlying set of the ordered blueprint $\bR=R^\mon$ is $\bR^\bullet=R$ and that the underlying set of the tropical hyperfield $\bT$ is $\bT^\bullet=\R_{\geq0}$. Thus by definition, a morphism $\bv:\bR\to\bT$ is a map $\bv^\bullet:R\to \R_{\geq0}$. 

That nonarchimedean seminorms can be interpreted as morphisms of hyperrings was observed Viro in \cite{Viro11}; also cf.\ \cite{Connes-Consani15} and \cite{Jun17}. In so far, the following theorem does not contain a novel mathematical fact, though its appearance in terms of ordered blueprints is new. Since it is a key fact for our theory, we include a short proof.
 
\begin{thm}\label{thm: nonarchimedean seminoms as morphisms}
 Let $R$ be a ring. The association $\bv\mapsto\bv^\bullet$ defines a bijection
 \[
  \begin{array}{cccc}
   \Phi: & \Hom(\bR,\bT) & \longrightarrow & \big\{ \, \text{nonarchimedean seminorms on }R \,\big\}. \\
         & \bv:\bR\to\bT & \longmapsto     & \bv^\bullet:R\to\R_{\geq0}
  \end{array}
 \]
\end{thm}

\begin{proof}
 We begin with the verification that $\Phi$ is well-defined, i.e.\ that $\bv^\bullet:R\to\R_{\geq0}$ is indeed a nonarchimedean seminorm. We denote $\bv^\bullet$ by $v$ in the following. Axioms \eqref{norm1} and \eqref{norm2} follow from the fact that $v=\bv^\bullet$ is a morphism of monoids with zero. Thus we are left with verifying axiom \eqref{norm3}.

 Let $a,b\in R$ and $c=a+b$. Then we have $c\leq a+b$ in $\bR$ and thus $\bv(c)\leq \bv(a)+\bv(b)$ in $\bT$. By Lemma \ref{lemma: characterization of monomial relations in the tropical hyperfield}, this means that the maximum among $\bv(a)$, $\bv(b)$ and $\bv(c)$ appears twice. This implies \eqref{norm3} at once.
 
 Conversely, consider a nonarchimedean seminorm $v:R\to\R_{\geq0}$ and define $\bv^\bullet=v$ as a map $\bR^\bullet=R\to\R_{\geq0}=\bT^\bullet$. By axioms \eqref{norm1} and \eqref{norm2}, $\bv^\bullet$ is a morphism of monoids with zero. Since $\bR^+=(\bR^\bullet)^+$ is the freely generated semiring, $\bv^\bullet$ extends uniquely to a semiring homomorphism $\bv^+:\bR^+\to\bT^+$. We are left with verifying that $\bv^+$ is order-preserving, which can be verified on generators $c\leq a+b$ of the partial order $\leq$ of $\bR$.
 
 The relation $a\leq b_1+b_2$ means that $a=b_1+b_2$ in $R$. By Lemma \ref{lemma: strong version of the strict triangular inequality}, the maximum among $v(a)$, $v(b_1)$ and $v(b_2)$ appears twice. By Lemma \ref{lemma: characterization of monomial relations in the tropical hyperfield}, this implies that $\bv(a)\leq\bv(b_1)+\bv(b_2)$, which shows that $\bv:\bR\to\bT$ is a morphism of ordered blueprints. This verifies the claim of the lemma.
\end{proof}

Given a nonarchimedean seminorm $v:R\to\R_{\geq0}$, we shall call the morphism $\bv:\bR\to\bT$ with $\bv^\bullet=v$ the \emph{associated morphism} in the following.

\begin{cor}\label{cor: k-linear seminorms as morphisms}
 Let $k$ be a field and $\bk=k^\mon$ the associated monomial ordered blueprint. Let $v:k\to\R_{\geq0}$ be a nonarchimedean absolute value and $\bv:\bk\to\bT$ the associated morphism. Let $f:k\to R$ be a $k$-algebra and $f^\mon:\bk\to \bR$ the associated $\bk$-algebra where $\bR=R^\mon$. Then the association $\bw\mapsto \bw^\bullet$ defines a bijection
 \[
  \begin{array}{cccc}
   \Phi: & \Hom_\bk(\bR,\bT) & \longrightarrow & \big\{ \, \text{nonarchimedean seminorms }w:R\to\R_{\geq0} \, \big| \, w\circ f=v \,\big\}.
  \end{array}
 \]
\end{cor}

\begin{proof}
 By Theorem \ref{thm: nonarchimedean seminoms as morphisms}, the association $\bw\mapsto \bw^\bullet$ defines a bijection between morphisms $\bw:\bR\to\bT$ and nonarchimedean seminorms $w:R\to\R_{\geq0}$. A morphism $\bw:\bR\to\bT$ is $\bk$-linear if $\bw\circ f^\mon=\bv$. This condition is evidently equivalent with $w\circ f=v$, which proves our claim.
\end{proof}


\subsection{Affine ordered blue schemes}
\label{subsection: affine ordered blue schemes}

For the purpose of this text, we define the category of \emph{affine ordered blue schemes} as the dual category $\OBAff$ of $\OBlpr$. We denote the anti-equivalences between these categories by
\[
 \begin{tikzcd}[column sep=2cm]
  \OBlpr \ar[shift left=0.5ex]{r}{\Spec} & \OBAff \ar[shift left=0.5ex]{l}{\Gamma}.
 \end{tikzcd}
\]
Typically, we say that $X=\Spec B$ is an affine ordered blue scheme where we assume implicitly that $B$ is an ordered blueprint. Given a morphism $f:B\to C$ of ordered blueprints, we write $f^\ast:Y\to X$ for the dual morphism from $Y=\Spec C$ to $X=\Spec B$. Given a morphism $\varphi:Y\to X$, we denote its dual morphism by $\varphi^\sharp:B\to C$. 

Let $k$ be an ordered blueprint. An \emph{affine ordered blue $k$-scheme} is an affine ordered blue scheme $X=\Spec B$ together with a morphism $\pi:X\to\Spec k$, which we call the \emph{structure morphism}. Often we suppress the structure morphism from the notation and refer to $X$ as an affine ordered blue $k$-scheme. A $k$-linear morphism between affine ordered blue $k$-schemes $X$ and $Y$ is a morphism $X\to Y$ that commutes with the structure morphisms of $X$ and $Y$. Note that the morphism $\pi:X\to \Spec k$ is the dual of an ordered blueprint morphism $f:k\to B$. This means that $B$ is a $k$-algebra.

\begin{ex}\label{ex: affine space and torus}
 Let $B$ be an ordered blueprint. We define the \emph{$n$-dimensional affine space over $B$} as the affine ordered blue $B$-scheme 
 \[
  \A^n_B \ = \ \Spec\big( \, B[T_1,\dotsc,T_n] \, \big)
 \]
 and the \emph{$n$-dimensional torus over $B$} as 
 \[
  \G_{m,B}^n \ = \ \Spec\big( \, B[T_1^{\pm1},\dotsc,T_n^{\pm1}] \, \big).
 \]
 By the universal property of $\B[T_1,\dotsc,T_n]$, a $B$-linear morphism $f:B[T_1,\dotsc,T_n]\to B$ corresponds to the tuple $\big(f(T_1),\dotsc,f(T_n)\big)$ in $B^n$. This establishes a canonical bijection
 \[
  \A_B^n(B) \ = \ \Hom_B\big(B[T_1,\dotsc, T_n],B\big) \ \stackrel\sim\longrightarrow \ B^n,
 \]
 which is analogous to the characterizing property of the $n$-dimensional affine space in usual algebraic geometry. Similarly, we have a canonical bijection
 \[
  \G_{m,B}^n(B) \ = \ \Hom_B\big(B[T_1^{\pm1},\dotsc,T_n^{\pm1}],B\big) \ \stackrel\sim\longrightarrow \ (B^\times)^n
 \]
 where $B^\times$ is the group of invertible elements in $B^\bullet$.
\end{ex}

\begin{rem}
 The reason that we restrict ourselves to affine ordered blue schemes is purely a matter of exposition. While it is possible to define affine ordered blue schemes as objects of the dual category of $\OBlpr$, the definition of ordered blue schemes requires a more sophisticated setup. The definition of an ordered blue scheme can be found in \cite{Baker-Lorscheid18} and \cite{Lorscheid15}. An alternative, but equivalent, definition uses relative schemes in the sense of To\"en and Vaqui\'e (\cite{Toen-Vaquie09}); see \cite{Lorscheid17} for a proof of the equivalence in the case of algebraic blueprints.
\end{rem}


\subsection{Scheme theoretic tropicalization}
\label{subsection: scheme theoretic tropicalization}

Let $k$ be a field and $\bk=k^\mon$. Let $v:k\to\R_{\geq0}$ be a nonarchimedean absolute value and $\bv:\bk\to\bT$ the associated morphism, cf.\ Theorem \ref{thm: nonarchimedean seminoms as morphisms}. Let $Y=\Spec B$ be an affine ordered blue $\bk$-scheme and $\bk\to B$ the structure map. The \emph{scheme theoretic tropicalization of $Y$ along $\bv$} is the affine ordered blue $\bT$-scheme $\Trop_\bv(Y)=\Spec\big(B\otimes_{\bk}\bT\big)$.

Note that the canonical inclusion $\bT\to B\otimes_{\bk}\bT$ endows $\Trop_\bv(Y)$ with the structure of an ordered blue $\bT$-scheme. Note further that since $-\otimes_{\bk}\bT$ is a functor, $\Trop_\bv(Y)$ is functorial in $Y$, i.e.\ a $k$-linear morphism $Y\to Y'$ induces a $\bT$-linear morphism $\Trop_\bv(Y)\to\Trop_\bv(Y')$.

\begin{ex}\label{ex: scheme theoretic tropicalizations of affine spaces and tori}
 Since $\bk[T_1,\dotsc,T_n]\otimes_\bk\bT=\bT[T_1,\dotsc,T_n]$, cf.\ Example \ref{ex: tensor products of ordered blueprints}, the scheme theoretic tropicalization of affine $n$-space over $\bk$ is $\Trop_\bv(\A_\bk^n)=\A_\bT^n$. Similarly, we have $\Trop_\bv(\G_{m,\bk}^n)=\G_{m,\bT}^n$.
\end{ex}


\subsection{Topology for rational point sets}
\label{subsection: topology for rational point sets}

Let $X=\Spec B$ be an affine ordered blue $\bT$-scheme. The \emph{set of $\bT$-rational points} is the set $X(\bT)=\Hom_\bT(B,\bT)$ of $\bT$-linear morphisms from $B$ to $\bT$. The \emph{affine topology of $X(\bT)$} is defined as the compact-open topology on $\Hom_\bT(B,\bT)$ with respect to the Euclidean topology of $\bT$ and the discrete topology for $B$. In other words, $X(\bT)$ has the coarsest topology such that the evaluation maps
\[
 \begin{array}{cccc}
  \ev_a: & X(\bT)    & \longrightarrow & \bT \\
         & f:B\to\bT & \longmapsto     & f(a)
 \end{array}
\]
are continuous for all $a\in B$.

The affine topology for $\bT$-rational point sets is very well-behaved, similar to the situation of rational point sets over topological fields.

\begin{thm}\label{thm: properties of rational point sets}
 The affine topology satisfies the following properties for affine ordered blue $\bT$-schemes $X$ and $Y$.
 \begin{enumerate}[label={(T\arabic*)}]
  \item\label{T1} The canonical bijection $\A^1_\bT(\bT)\to\bT$ is a homeomorphism.
  \item\label{T2} The canonical bijection $(X\times Y)(\bT)\to X(\bT)\times Y(\bT)$ is a homeomorphism.
  \item\label{T3} A $\bT$-linear morphism $\varphi:Y\to X$ induces continuous map $\varphi_\bT:Y(\bT)\to X(\bT)$. If $\varphi$ is an open (a closed) immersion, then $\varphi_\bT$ is an open (a closed) topological embedding.
 \end{enumerate}
\end{thm}

\begin{proof}
 Since $B$ is a topological ordered blueprint (i.e.\ the multiplication map $\bT\times\bT\to\bT$ is continuous) with open unit group (i.e.\ $\bT^\times=\R_{>0}$ is open in $\bT$ and multiplicative inversion defines a continuous map $\bT^\times\to\bT^\times$), everything follows from \cite[Thm.\ 6.4]{Lorscheid15} but the claim that a closed immersion induces a closed topological embedding. We will prove this claim in the following.
 
 Let $\varphi:Y\to X$ be a closed $\bT$-linear immersion. By \cite[Thm.\ 6.4 (F5)]{Lorscheid15}, $\varphi_\bT$ is a topological embedding. All that is left to prove is that the image $\varphi_\bT(Y(\bT))$ is closed in $X(\bT)$.
 
 If $C=\bpgenquot{B}{S}$ for some set of relations $S=\{x_l\leq y_l\}$ on $B^+$, then the image $\varphi_\bT(Y(\bT))$ equals the subset
 \[\textstyle
  \big\{ f:B\to\bT \big| \sum f(a_i)\leq \sum f(b_j)\text{ for every }\sum a_i\leq\sum b_j\text{ in }S\big\}
 \]
 of $X(\bT)=\{f:B\to\bT\}$. In the following proof, we break down the relations $\sum f(a_i)\leq \sum f(b_j)$ in $\bT$ into more elementary terms and trace this back to an expression of $\varphi_\bT(Y(\bT))$ as an intersection of finite unions of elementary closed sets of $X(\bT)$ of the form $U_{a,Z}=\{f:B\to\bT|f(a)\in Z\}$ where $a\in B$ and $Z\subset\bT$ is a closed subset. Note that $U_{a,Z}$ is indeed closed since the complement of $U_{a,Z}$ in $X(\bT)$ is the elementary open subset $U_{a,Z^c}$ of $X(\bT)$ where the complement $Z^c=\bT-Z$ of $Z$ is open in $\bT$.
 
 Our first observation is that 
 \[
  \varphi_\bT(Y(\bT)) \ = \ \bigcap_{\sum a_i\leq\sum b_j\text{ in }S} \ \textstyle \big\{ f:B\to\bT \big| \sum f(a_i)\leq \sum f(b_j)\big\},
 \]
 which reduces us to the proof that a subset of the form $\big\{ f:B\to\bT \big| \sum f(a_i)\leq \sum f(b_j)\big\}$ is closed in $X(\bT)$. Therefore let us focus on one fixed relation $\sum_{i\in I}a_i\leq \sum_{j\in J}b_j$ in the following where we specify the index sets $I$ and $J$ for reference.
 
 Since $\bT$ is monomial, i.e.\ generated by relations of the form $c\leq \sum d_l$, there must be a partition $J=\coprod_{i\in I}J_i$ of $J$ such that $a_i\leq \sum_{j\in J_i}b_j$ for all $i\in I$. This means that
 \[
  {\textstyle \big\{ f:B\to\bT \big| \sum f(a_i)\leq \sum f(b_j)\big\} } \ = \bigcup_{J=\coprod J_i} \ \bigg( \bigcap_{i\in I} \ \textstyle \big\{f:B\to\bT\big|f(a_i)\leq \sum_{j\in J_i}f(b_j)\big\} \bigg)
 \]
 is the finite union over all partitions of $J$ of intersections of subsets of the form $\big\{f:B\to\bT\big|f(a_i)\leq \sum_{j\in J_i}f(b_j)\big\}$. We have reduced the proof therefore to the situation of showing that a subset of the form $\big\{f:B\to\bT\big|f(a_i)\leq \sum_{j\in J_i} f(b_j)\big\}$ is closed in $X(\bT)$. For simplicity, we assume $J_i=\{1,\dotsc,n\}$ and set $b_0=a$.
 
 In the case that $J_i$ is empty, the relation in question is $a\leq 0$. Thus we could also equally assume that $J_i=\{1\}$ and $b_1=0$. This allows us to rely on Lemma \ref{lemma: characterization of monomial relations in the tropical hyperfield}, which states that $f(b_0)\leq\sum f(b_j)$ holds in $\bT$ if and only if the maximum occurs twice among $f(b_0),\dotsc,f(b_n)$, i.e.\ if we have $f(b_k)=f(b_l)\geq f(b_j)$ for some $0\leq k<l\leq n$ and all $j=0,\dotsc,n$. This means that $\big\{f:B\to\bT\big|f(b_0)\leq \sum_{j=1}^n f(b_j)\big\}$ equals
 \[
   \bigcup_{0\leq k<l\leq n} \ \bigg( \big\{f:B\to\bT\big|f(b_k)=f(b_l)\big\} \quad \cap \quad \bigcap_{j=0}^n \big\{f:B\to\bT\big|f(b_l)\geq f(b_j)\big\} \bigg),
 \]
 i.e.\ a finite union of the intersection of subsets of the forms $\big\{f:B\to\bT\big|f(b_k)=f(b_l)\big\}$ and $\big\{f:B\to\bT\big|f(b_l)\geq f(b_j)\big\}$ of $X(\bT)$. This reduces our proof to the study of these two particular types of subsets.
 
 We begin with $\big\{f:B\to\bT\big|f(b_k)=f(b_l)\big\}$, which can be expressed as $\big\{f:B\to\bT\big|(f(b_k),f(b_l))\in\Delta\big\}$ where $\Delta\in\bT\times\bT$ is the diagonal. Since $\bT$ is Hausdorff, $\Delta$ is closed in $\bT\times\bT$ and can thus be written as an intersection of finite unions of basic closed subsets, i.e.\
 \[
  \Delta \ = \ \bigcap_{p\in P} \ \bigcup_{q\in Q_p} V_{p,q,k}\times V_{p,q,l}
 \]
 where $P$ and $Q_p$ are index sets, with $Q_p$ finite for every $p\in P$, and where $V_{p,q,k}$ and $V_{p,q,l}$ are closed subsets of $\bT$. Thus 
 \[
  \big\{f:B\to\bT\big|(f(b_k),f(b_l)\in\Delta\big\} \ = \ \bigcap_{p\in P} \ \bigcup_{q\in Q_p} \Big(U_{b_k,V_{p,q,k}}\cap U_{b_l,V_{p,q,l}} \Big)
 \]
 where $U_{b_j,V_{p,q,j}}=\big\{f:B\to\bT\big|f(b_j)\in V_{p,q,j}\big\}$ is a basic closed subset of $\bT$ for $j\in\{k,l\}$. This shows that $\big\{f:B\to\bT\big|f(b_k)=f(b_l)\big\}$ is a closed subset of $X(\bT)$.
 
 We continue with the remaining case $\big\{f:B\to\bT\big|f(b_l)\geq f(b_j)\big\}$, which equals the set $\big\{f:B\to\bT\big|(f(b_k),f(b_j)\in\nabla\big\}$ where $\nabla$ is the subset of all $(c,d)\in\bT\times\bT$ with $c\geq d$. Since $\nabla$ is closed in $\bT\times\bT$, an analogous argument as in the preceding case shows that $\big\{f:B\to\bT\big|f(b_l)\geq f(b_j)\big\}$ is a closed subset of $X(\bT)$. This concludes the proof that $\varphi_\bT:Y(\bT)\to X(\bT)$ is a closed topological embedding.
\end{proof}

\begin{rem}
 In \cite{Lorscheid15}, the affine topology for $X(\bT)$ is extended to possibly non-affine ordered blue $\bT$-schemes $X$. In this more general situation, the topology on $X(\bT)$ is called the \emph{fine topology}. By \cite[Thm.\ 6.4]{Lorscheid15}, we conclude that for every affine open covering $\{U_i\}$ of $X$, we have $X(\bT)=\bigcup U_i(\bT)$ as sets and that the inclusions $U_i(\bT)\to X(\bT)$ are open topological embeddings, which determines the fine topology of $X(\bT)$ in terms of the affine topology for the $U_i(\bT)$. In particular the affine and the fine topology agree on $X(\bT)$ if $X$ is affine.
\end{rem}

\begin{ex}\label{ex: topologies of tropical plane and tropical plane line}
 Since $\bT[T_1,T_2]=\bT[T_1]\otimes_\bT\bT[T_2]$, we have $\A_\bT^2=\A_\bT^1\times_\bT\A_\bT^1$. Thus by \ref{T1} and \ref{T2} of Theorem \ref{thm: properties of rational point sets}, we have
 \[
  \A_\bT^2(\bT) \ = \ \A_\bT^1(\bT)\times\A_\bT^1(\bT) \ = \ \bT \times \bT \ = \ \bT^2
 \]
 as topological spaces, where we consider $\bT^2$ with respect to the product topology.
 
 Let $B=\bpgenquot{\bT[T_1,T_2]}{0\leq T_1+T_2+1}$ and $X=\Spec B$. By \ref{T3} of Theorem \ref{thm: properties of rational point sets},
 \[
  X(B) \ = \ \Hom_\bT(B,\bT) \ = \ \big\{ \, (a_1,a_2)\in\bT^2 \, \big| \, 0\leq a_1+a_2+1 \, \big\}
 \]
 is a closed subspace of $\bT^2$. Since $0\leq a_1+a_2+1$ if and only if the maximum occurs twice, we see that $X(B)$ is the standard plane tropical line, which can be depicted as 
 \[
  \beginpgfgraphicnamed{tikz/fig1}
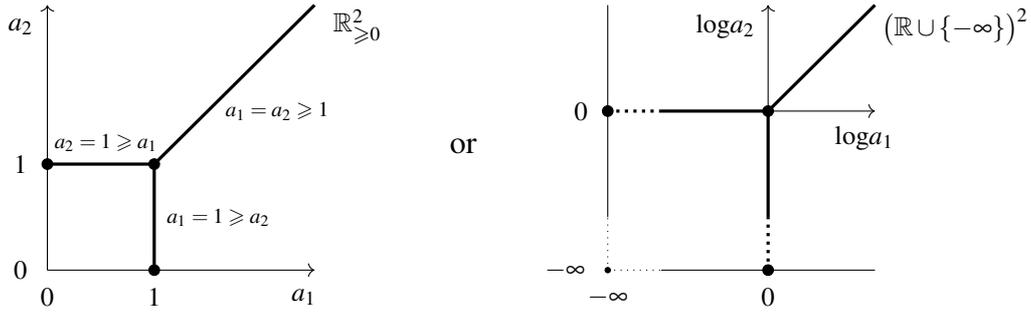
  \begin{tikzpicture}[inner sep=0,x=20pt,y=20pt,font=\footnotesize]
   \draw[->] (0,0) -- (0,5);
   \draw[->] (0,0) -- (5,0);
   \draw[very thick] (0,2) -- (2,2);
   \draw[very thick] (2,0) -- (2,2);
   \draw[very thick] (5,5) -- (2,2);
   \filldraw (2,0) circle (2pt);
   \filldraw (0,2) circle (2pt);
   \filldraw (2,2) circle (2pt);
   \node at (0,-0.5) {$0$};
   \node at (2,-0.5) {$1$};
   \node at (4.8,-0.5) {$a_1$};
   \node at (-0.5,0) {$0$};
   \node at (-0.5,2) {$1$};
   \node at (-0.5,4.6) {$a_2$};
   \node at (5.8,4.6) {$\R_{\geq0}^2$};
   \node[font=\tiny] at (1.1,2.4) {$a_2=1\geq a_1$};
   \node[font=\tiny] at (3.2,1) {$a_1=1\geq a_2$};
   \node[font=\tiny] at (4.3,3) {$a_1=a_2\geq 1$};
  \end{tikzpicture}
 \hspace{0.9cm}
  \begin{tikzpicture}[inner sep=0,x=20pt,y=20pt]
   \node at (0,0) {};
   \node at (0,3) {or};
  \end{tikzpicture}
 \hspace{0.9cm}
  \begin{tikzpicture}[inner sep=0,x=20pt,y=20pt,font=\footnotesize]
   \draw[->] (3,1) -- (3,5);
   \draw[->] (1,3) -- (5,3);
   \draw (0,1) -- (0,5);
   \draw (1,0) -- (5,0);
   \draw[dotted] (0,0) -- (0,1);
   \draw[dotted] (0,0) -- (1,0);
   \filldraw (0,0) circle (1pt);
   \draw[very thick] (1,3) -- (3,3);
   \draw[very thick] (3,1) -- (3,3);
   \draw[very thick] (5,5) -- (3,3);
   \draw[very thick,dotted] (3,0) -- (3,1);
   \draw[very thick,dotted] (0,3) -- (1,3);
   \filldraw (3,0) circle (2pt);
   \filldraw (0,3) circle (2pt);
   \filldraw (3,3) circle (2pt);
   \node at (0,-0.5) {$-\infty$};
   \node at (3,-0.5) {$0$};
   \node at (4.8,2.5) {$\log a_1$};
   \node at (-0.8,0) {$-\infty$};
   \node at (-0.5,3) {$0$};
   \node at (2.2,4.6) {$\log a_2$};
   \node at (6.5,4.6) {$\big(\R\cup\{-\infty\}\big)^2$};
  \end{tikzpicture}
  \endpgfgraphicnamed
 \]
 where the illustration on the left hand side uses the natural coordinates $(a_1,a_2)$ in $\bT^2=\R_{\geq0}^2$ and the illustration on the right hand side follows the more common convention of double-logarithmic coordinates $(\log a_1,\log a_2)$ in $\big(\R\cup\{-\infty\}\big)^2$.
\end{ex}


\section{The Kajiwara-Payne tropicalization as a rational point set}
\label{section: recovering the Kajiwara-Payne tropicalization}

We begin this section with a review of Berkovich spaces and the Kajiwara-Payne tropicalization before we explain how to recover them as rational point sets of a scheme theoretic tropicalization.


\subsection{The Berkovich analytification}
\label{subsection: the Berkovich analytification}

Let $k$ be a field together with a nonarchimedean absolute value $v:k\to \R_{\geq0}$. Let $f:k\to R$ be a $k$-algebra and $X=\Spec R$. As a topological space, the \emph{Berkovich space of $X$} is the set
\[
 X^\an \ = \ \big\{\text{nonarchimedean seminorms }w:R\to\R_{\geq0}\text{ such that }w\circ f=v \big\}
\]
together with the compact-open topology with respect to the discrete topology for $R$ and the Euclidean topology of $\R_{\geq0}$. In other words, the topology of $X^\an$ is the coarsest topology such that the maps
\[
 \begin{array}{cccc}
   \ev_a: & X^\an            & \longrightarrow & \R_{\geq0} \\
          & w:R\to\R_{\geq0} & \longmapsto     & w(a)
 \end{array}
\]
are continuous for all $a\in R$. Thus the topology of $X^\an$ is generated by the open subsets of the form
\[
 U_{a,V} \ = \ \big\{ \, w:R\to\R_{\geq0} \, \big| \, w(a)\in V \, \big\}
\]
where $a\in R$ and $V\subset\R_{\geq0}$ is an open subset.

\begin{ex}\label{ex: Berkovich analytification}
 Berkovich analytifications tend to be rather involved topological spaces. In the case of curves, one can find a description in Berkovich's book \cite{Berkovich90}. For the purpose of illustration, we will give the description of an easy case of a Berkovich analytification, which occurs for the trivial absolute value $v:k\to\R_{\geq0}$ with $v(a)=1$ for all $a\in k^\times$ and the affine line $X=\A^1_k=\Spec k[T]$.
 
 The nonarchimedean seminorms $w:k[T]\to\R_{\geq0}$ extending the trivial absolute value $v:k\to\R_{\geq0}$ are classified as follows:
 \begin{itemize}
  \item the \emph{trivial norm} $w_0$ with $w_0(g)=1$ whenever $g\neq 0$;
  \item for every irreducible polynomial $f\in k[T]$ and every $r\in (0,1)$ the \emph{$f$-adic norm} $w_{f,r}$ with $w_{f,r}(f^ig/h)=r^i$ whenever $gh$ is not divisible by $f$;
  \item for every irreducible polynomial $f\in k[T]$ the seminorm $w_{f,0}$ with $w_{f,0}(g)=0$ if $g$ is divisible by $f$ and $w_{f,0}(g)=1$ if not;
  \item for every $r\in (0,1)$ the \emph{$\infty$-adic norm} $w_{\infty,r}$ with $w_{\infty,r}(g/h)=r^{\deg h-\deg g}$ for $gh\neq0$.
 \end{itemize}
 The analytification $X^\an$ can be depicted as 
 \[
  \beginpgfgraphicnamed{tikz/fig2}
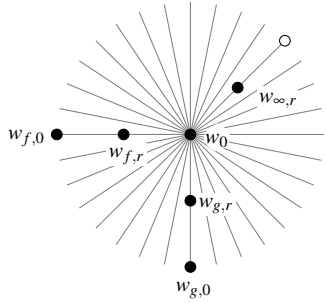
  \begin{tikzpicture}[inner sep=0,x=25pt,y=25pt,font=\tiny]
   \node (origin) {};
   \foreach \a in {1,...,32}{\draw[gray] (0,0) -- ++(\a*360/32:2);}
   \filldraw (0,0) circle (2pt);
   \filldraw (45:1) circle (2pt);
   \filldraw (180:1) circle (2pt);
   \filldraw (270:1) circle (2pt);
   \filldraw (180:2) circle (2pt);
   \filldraw (270:2) circle (2pt);
   \draw[fill=white] (45:2) circle (2pt);
   \node[rounded corners=7pt,fill=white,text height=1.4ex,text depth=0.6ex,font=\tiny] at (0.4,-0.05) {$\, w_0\, $};
   \node[rounded corners=0pt,fill=white,text height=1.1ex,text depth=0.6ex,font=\tiny] at (1.3,0.55) {$\,w_{\infty,r}$};
   \node[rounded corners=0pt,fill=white,text height=1.0ex,text depth=0.6ex,font=\tiny] at (-0.95,-0.3) {$\, w_{f,r}\,$};
   \node[rounded corners=0pt,fill=white,text height=1.0ex,text depth=0.6ex,font=\tiny] at (0.4,-1.05) {$w_{g,r}\,$};
   \node at (-2.45,-0.05) {$w_{f,0}$};
   \node at (0.05,-2.35) {$w_{g,0}$};
  \end{tikzpicture}
  \endpgfgraphicnamed
 \]  
 where $f$ and $g$ are irreducible polynomials in $k[T]$ and where one has to imagine an infinite number of rays emerging from the center $w_0$. The empty circle at the end of the ray of $w_{\infty,r}$ indicates the missing point at infinity. The topology of $X^\an$ is generated by open subsets of the following forms:
 \begin{itemize}
  \item $U_{f,I}=\{w_{f,r}|r\in I\}$ where $f\in k[T]$ is irreducible and $I\subset (0,1]$ is open or $f=\infty$ and $I\subset (0,1)$ is open;
  \item $V_{f,I}=X^\an-\{w_{f,r}|r\in I\}$ where $f\in k[T]$ is irreducible and $I\subset (0,1]$ is closed or $f=\infty$ and $I\subset (0,1)$ is closed.
 \end{itemize}
\end{ex}


\subsection{The Kajiwara-Payne tropicalization}
\label{subsection: the Kajiwara-Payne tropicalization}

The tropicalization of an affine $k$-scheme $X=\Spec R$ along a nonarchimedean absolute value $v:k\to \R_{\geq0}$ requires an additional choice of coordinates. Such a choice is given by a closed $k$-linear immersion $\iota:X\to T$ into an affine toric variety $T$, which is an affine $k$-scheme of the form $T=\Spec k[A]$ for a suitable multiplicative and commutative monoid $A$. To be precise, $T$ is a toric variety if $A$ satisfies the following conditions:
\begin{enumerate}
 \item $A$ is  finitely generated as a monoid;
 \item $A$ is \emph{integral}, i.e.\ $A$ can be embedded as a submonoid in a group $G$;
 \item $A$ is \emph{saturated}, i.e.\ if $A^\gp$ is the subgroup of $G$ generated by $A$ and $a^n\in A$ for some $a\in A^\gp$ and $n\geq 1$, then $a\in A$.
\end{enumerate}

The closed immersion $\iota:X\to T$ corresponds to a surjective $k$-algebra homomorphism $\iota^\sharp:k[A]\to R$. The \emph{ideal of definition for $X$} is the ideal $I=\ker(\iota^\sharp)$ of $k[A]$. The restriction of $\iota^\sharp$ to $A$ yields a morphism $\iota^\sharp_A:A\to R$ of multiplicative monoids. 

Consider an element $p=\sum c_a a$ in $k[A]$ where $c_a\in k$ and $a\in A$. We define $p^\trop=\sum v(c_a)a$ as a finite formal $\R_{\geq0}$-linear combinations of elements of $A$. We define $I^\trop=\{p^\trop|p\in I\}$. 

Let $\R_{\geq0}^A=\Hom(A,\R_{\geq0})$ be the set of maps $f:A\to\R_{\geq0}$ with $f(1)=1$ and $f(ab)=f(a)f(b)$ for all $a,b\in A$. It comes together with the compact-open topology with respect to the discrete topology for $A$ and the Euclidean topology of $\R_{\geq0}$. It is generated by the open subsets of the form
\[
 U_{a,V} \ = \ \big\{ \, f\in\R_{\geq0}^A \, \big| \, f(a)\in V \, \big\}
\]
where $a\in A$ and $V\subset\R_{\geq0}$ is an open subset.

Let $p=\sum c_aa$ be a finite $\R_{\geq0}$-linear combination. The \emph{zero set} or \emph{bend locus of $p$} is the subset
\[
 Z(p) \ = \ \big\{ \, f\in\R_{\geq0}^A \, \big| \, \text{the maximum occurs twice in }\{c_af(a)\}_{a\in A} \, \big\}
\]
of $\R_{\geq0}^A$. It is a closed subset of $\R_{\geq0}^A$ since it can be written as the intersection of closed analytic subsets of $\R_{\geq0}^A$. We consider $Z(p)$ together with the subspace topology.

The \emph{tropicalization of $X$ along $v:k\to\R_{\geq0}$ with respect to $\iota:X\to T$} is the closed subspace
\[
 X^\trop \ = \ \bigcap_{p\in I^\trop} \, Z(p)
\]
of $\R_{\geq0}^A$.

The \emph{Kajiwara-Payne tropicalization of $X$ along $v:k\to\R_{\geq0}$ with respect to $\iota:X\to T$} is the map 
\[
 \begin{array}{cccc}
  \trop: & X^\an      & \longrightarrow & X^\trop. \\
         &  w:R\to\R_{\geq0}  & \longmapsto     & w\circ \iota^\sharp_A
 \end{array}
\]
Note that the composition $f=w\circ \iota^\sharp_A$ is obviously in $\R_{\geq0}^A=\Hom(A,\R_{\geq0})$. If $0=\sum c_aa$ in $R$ where $a\in A$ and $c_a\in k$, then by Lemma \ref{lemma: strong version of the strict triangular inequality} the maximum among the $v(c_a)w(a)$ is assumed twice. Thus $f$ lies indeed in $X^\trop$. Note further that the map $\trop:X^\an\to X^\trop$ is continuous, which is immediately clear from the definitions of the respective topologies as compact-open topologies.

\begin{rem}\label{rem: the Kajiwara-Payne tropicalization map is surjective}
 The Kajiwara-Payne tropicalization $\trop:X^\an\to X^\trop$ is moreover proper and surjective, due to the following argument that was explained to the author by Sam Payne. In the case that $v$ is trivial or that $k$ algebraically closed and complete with respect to $v$, then this is proven in \cite{Payne09}. 
 
 If $k$ is arbitrary with non-trivial absolute value $v:k\to\R_{\geq0}$, then we can reduce the claim to the corresponding claim of a suitable field extension of $k$ by the following arguments. We write $X_K$ for the base change of $X$ to a field extension $K$ of $k$. If $\hat{k}$ is the completion of $k$ with respect to $v$ and $\hat v$ is the canonical extension of $v$ to $\hat k$, then the Berkovich spaces $X^\an$ and $(X_{\hat k})^\an$ agree since every nonarchimedean seminorm $R\to \R_{\geq0}$ extending $v$ extends uniquely to a nonarchimedean seminorm on $\hat R=R\otimes_k\hat k$ that extends $\hat v$.
 
 Since $\hat k$ is complete, the absolute value $\hat v$ extends uniquely to the algebraic closure $K$ of $\hat k$. As shown in \cite{Berkovich90}, $(X_{\hat k})^\an$ is the quotient of $(X_K)^\an$ by the action of the Galois group of $K$ over $\hat k$. Completing $K$ yields an algebraically closed and complete field with $(X_{\hat K})^\an=(X_K)^\an$, to which Payne's result from \cite{Payne09} applies. Thus we obtain a proper and surjective map
 \[
  (X_{\hat K})^\an \ = \ (X_K)^\an \ \longrightarrow \ (X_{\hat k})^\an \ = \ X^\an \ \stackrel{\trop}\longrightarrow \ X^\trop.
 \]
 Clearly this implies that $\trop$ is surjective. Given a compact subset $Z$ of $X^\trop$, then its inverse image $Z'$ in $(X_{\hat K})^\an$ is compact as the inverse image under a proper map. Since the projection $(X_{\hat K})^\an\to X^\an$ is continuous, the image $Z''$ of $Z'$ is compact in $X^\an$ is compact. Since $Z''=\trop^{-1}(Z)$, this shows that $\trop$ is proper. 
\end{rem}

\begin{rem}\label{rem: corner locus as zero set using the tropical hyperfield}
 For the following reinterpretation of $Z(p)$, we identify $\R_{\geq0}=\bT$ as sets. Let $f\in\R_{\geq0}^A$ and $p=\sum c_aa\in\N[A]^+$. By Lemma \ref{lemma: characterization of monomial relations in the tropical hyperfield}, the maximum among the terms $\{c_af(a)\}$ occurs twice if and only if $0\leq \sum c_af(a)$ as elements of $\bT^+$. Thus the identification $\R_{\geq0}=\bT$ leads to an identification
 \[\textstyle
  Z(p) \ = \ \big\{ \, f\in\bT^A \, \big| \, 0\leq \sum c_af(a) \,\big\}.
 \]
 This expression for the zero set of $p=\sum c_aa$ stays in direct analogy to the zero set of a function $p\in k^A$ over a field $k$. One can find a similar expression for $Z(p)$ using Izhakian's extended tropical semiring in \cite{Izhakian-Rowen10}; also cf.\ Remark \ref{rem: extended tropical numbers}.
\end{rem}

\begin{ex}\label{ex: Kajiwara-Payne tropicalization}
 We explain the Kajiwara-Payne tropicalization in the example of the trivial absolute value $v:k\to\R_{\geq0}$ and the line in $\A^2_k=\Spec k[T_1,T_2]$ given by $T_1+T_2+1$, i.e.\ $X=\Spec R$ with $R=k[T_1,T_2]/I$ where the ideal of definition is $I=(T_1+T_2+1)$.
 
 Then $X$ is isomorphic to $\A^1_k$ for which we have calculated the Berkovich analytification in Example \ref{ex: Berkovich analytification}. The tropicalization $X^\trop$ is the standard plane tropical line as illustrated in Example \ref{ex: topologies of tropical plane and tropical plane line}. 
 
 Let $A=\{T_1^{e_1}T_2^{e_2}|e_1,e_2\in\N\}$. Note that the map $\iota_A^\sharp: A \to k[T_1,T_2] \to R$ is injective. Since the classes of $T_1$ and $T_2$ in $R$ are irreducible, it follows immediately from the definition of the seminorm $w_{f,r}:R\to\R_{\geq0}$ in Example \ref{ex: Berkovich analytification} that the restriction $w_{f,r}\circ\iota_A^\sharp$ of $w_{f,r}$ to $A$ is trivial unless $f$ is one of $T_1$, $T_2$ or $\infty$. Thus the Kajiwara-Payne tropicalization can be illustrated as 
 \[
  \beginpgfgraphicnamed{tikz/fig3}
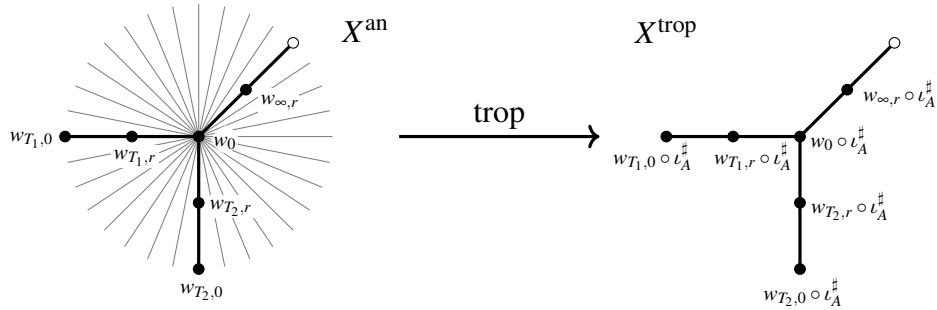
  \begin{tikzpicture}[inner sep=0,x=25pt,y=25pt,font=\tiny]
   \node (origin) {};
   \foreach \a in {1,...,32}{\draw[gray] (0,0) -- ++(\a*360/32:2);}
   \draw[very thick,black] (0,0) -- ++(45:2);
   \draw[very thick,black] (0,0) -- ++(270:2);
   \draw[very thick,black] (0,0) -- ++(180:2);
   \filldraw (0,0) circle (2pt);
   \filldraw (45:1) circle (2pt);
   \filldraw (180:1) circle (2pt);
   \filldraw (270:1) circle (2pt);
   \filldraw (180:2) circle (2pt);
   \filldraw (270:2) circle (2pt);
   \draw[fill=white] (45:2) circle (2pt);
   \node[rounded corners=7pt,fill=white,text height=1.4ex,text depth=0.6ex,font=\tiny] at (0.4,-0.05) {$\,w_0\,$};
   \node[rounded corners=0pt,fill=white,text height=1.1ex,text depth=0.6ex,font=\tiny] at (1.2,0.55) {$\,w_{\infty,r}$};
   \node[rounded corners=0pt,fill=white,text height=1.0ex,text depth=0.6ex,font=\tiny] at (-0.95,-0.3) {$\, w_{T_1,r}\,$};
   \node[rounded corners=0pt,fill=white,text height=1.0ex,text depth=0.6ex,font=\tiny] at (0.5,-1.05) {$w_{T_2,r}\,$};
   \node at (-2.5,-0.05) {$w_{T_1,0}$};
   \node at (0.05,-2.35) {$w_{T_2,0}$};
   \draw[very thick,->] (3,0) -- (6,0);
   \node[font=\normalsize] at (4.5,0.3) {$\trop$};
   \node[font=\normalsize] at (2.5,1.6) {$X^\an$};
   \node[font=\normalsize] at (7,1.6) {$X^\trop$};
   \draw[very thick] (9,0) -- ++(45:2);
   \draw[very thick] (9,0) -- ++(180:2);
   \draw[very thick] (9,0) -- ++(270:2);
   \filldraw (9,0) circle (2pt);
   \filldraw (9,0)++(45:1) circle (2pt);
   \filldraw (9,0)++(180:1) circle (2pt);
   \filldraw (9,0)++(270:1) circle (2pt);
   \filldraw (9,0)++(180:2) circle (2pt);
   \filldraw (9,0)++(270:2) circle (2pt);
   \draw[fill=white] (9,0)++(45:2) circle (2pt);
   \node at (9.6,-0.05) {$w_0\circ\iota_A^\sharp$};
   \node at (10.5,0.65) {$w_{\infty,r}\circ\iota_A^\sharp$};
   \node at (8.3,-0.3) {$w_{T_1,r}\circ\iota_A^\sharp$};
   \node at (9.75,-1.05) {$w_{T_2,r}\circ\iota_A^\sharp$};
   \node at (6.8,-0.3) {$w_{T_1,0}\circ\iota_A^\sharp$};
   \node at (9.05,-2.35) {$w_{T_2,0}\circ\iota_A^\sharp$};
  \end{tikzpicture}
  \endpgfgraphicnamed
 \] 
 where all thin rays of $X^\an$ are contracted to the central point $w_0\circ\iota_A^\sharp$ of $X^\trop$ and the three thick rays of $X^\an$ are mapped bijectively to the three rays of $X^\trop$.
\end{ex}


\subsection{Recovering the analytification and tropicalization as rational point sets}
\label{subsection: analytification and tropicalization as rational point sets}

We continue with the context of the previous sections and define $\bk=k^\mon$, $\bR=R^\mon$ and $\bX=\Spec\bR$. Let $\pi:\bk[A]\to\bR$ be the morphism that maps $ca$ to $\iota^\sharp(ca)$ where $c\in \bk=k$ and $a\in A$. We define
\[\textstyle
 B \ = \ \bpgenquot{\bk[A]}{c_aa\leq \sum c_jb_j|\pi(c_aa)\leq\sum \pi(c_jb_j)\text{ in }\bR},
\]
whose underlying monoid is the submonoid $B^\bullet=\{\pi(ca)|c\in k, a\in A\}$ of $R$, whose ambient semiring is $B^+=(B^\bullet)^+$ and whose partial order is generated by the relations $\pi(c_aa)\leq \sum \pi(c_jb_j)$ with $c_a,c_j\in k$ and $a,b_j\in B^\bullet$ for which $\iota^\sharp(c_aa)=\sum \iota^\sharp(c_jb_j)$ in $R$. We define $Y=\Spec B$. Note that the inclusion $B\to\bR$ defines a $\bT$-linear morphism $\beta:\Trop_\bv(\bX)\to \Trop_{\bv}(Y)$.

\begin{thm}\label{thm: recovering the KP tropicalization from the scheme theoretic tropicalization}
 There are canonical homeomorphisms $\Phi^\an:X^\an\to\Trop_\bv(\bX)(\bT)$ and $\Phi^\trop:X^\trop\to\Trop_\bv(Y)(\bT)$ such that the diagram
 \[
  \begin{tikzcd}[column sep=3cm]
   X^\an \ar{r}{\trop} \ar{d}[left]{\Phi^\an} & X^\trop \ar{d}[right]{\Phi^\trop} \\
   \Trop_\bv(\bX)(\bT) \ar{r}{\beta_\ast} & \Trop_{\bv}(Y)(\bT)
  \end{tikzcd}
 \]
 commutes. 
\end{thm}

\begin{proof}
 As a first step, we define the canonical map $\Phi^\an:X^\an\to\Trop_\bv(\bX)(\bT)$. Let $w:R\to\R_{\geq0}$ be a seminorm in $X^\an$ and $\bw:\bR\to\bT$ the associated morphism from Corollary \ref{cor: k-linear seminorms as morphisms}. We define the $\bT$-linear morphism $f=\Phi^\an(w):\bR\otimes_\bk\bT\to\bT$ as the morphism induced by $\bw$ and the identity $\bT\to\bT$. It sends an element $r\otimes t$ to $t\cdot \bw(r)$.
 
 Conversely, given a $\bT$-linear morphism $f:\bR\otimes_\bk\bT\to\bT$, we define $w=\Psi^\an(f)$ as the seminorm associated with the composition $\bR\to\bR\otimes_\bk\bT\to\bT$. Since the composition $\bk\to\bR\to\bR\otimes_\bk\bT\to\bT$ is equal to $\bv:\bk\to\bT$, the restriction of $w$ to $k$ is $v$. This defines a map $\Psi^\an:\Trop_\bv(\bX)(\bT)\to X^\an$. It is clear that the maps $\Phi^\an$ and $\Psi^\an$ are mutually inverse bijections.
 
 We continue with showing that both $\Phi^\an$ and $\Psi^\an$ are open. This can be verified on the generators of the topologies of $X^\an$ and $\Trop_\bv(\bX)(\bT)$. Consider such a generator $U_{r,V}=\{w:R\to\R_{\geq0}|w(r)\in V\}$ of the topology of $X^\an$ where $r\in R$ and $V\subset \R_{\geq0}$ is open. Then we can consider $r$ as an element of $\bR$ and $V$ as an open subset of $\bT$ and find that 
 \[
  \Phi^\an(U_{r,V}) \ = \ \big\{ \ f:\bR\otimes_\bk\bT\to\bT \ \big| \ f(r\otimes 1)\in V \ \big\} \ = \ U_{r\otimes 1,V},
 \]
 which is a generator of the topology of $\Trop_\bv(\bX)(\bT)$. Conversely, consider a generator $U_{r\otimes t,V}=\{f:\bR\otimes_\bk\bT\to\bT|f(r\otimes t)\in V\}$ of the topology of $\Trop_\bv(\bX)(\bT)$ where $r\otimes t\in\bR\otimes_\bk\bT$ and $V\subset \bT$ is open. 
 
 Since $f$ is $\bT$-linear, we have $f(r\otimes t)=tf(r\otimes 1)$, and $tf(r\otimes 1)\in V$ is equivalent with $f(r\otimes 1)\in V_t=\{a\in\bT|at\in V\}$. Since the multiplication in $\bT$ is continuous, $V_t$ is open in $\bT$. Considering $r$ as an element in $R$ and $V_t$ as an open subset of $\R_{\geq0}$ yields that 
 \[
  \Psi^\an(U_{r\otimes t,V}) \ = \ \big\{ \ w:R\to\R_{\geq0} \ \big| \ w(r)\in V_t \ \big\} \ = \ U_{r,V_t},
 \]
 which is a generator of the topology of $X^\an$. This concludes the proof that $\Phi^\an$ is a homeomorphism.
 
 We turn to the definition of the homeomorphism $\Phi^\trop:X^\trop\to\Trop_\bv(Y)(\bT)$. Consider a map $f:A\to\R_{\geq0}$ in $X^\trop$. We define $g=\Phi^\trop(f):B\otimes_\bk\bT\to\bT$ by the rule $g(ca\otimes t)=\bv(c)f(a)t$ for $c\in k$, $a\in A$ and $t\in\bT$.
 
 This association is well-defined as a map since 
 \[
  g(ca\otimes t) \ = \ \bv(c)f(a)t \ = \ f(a)\cdot (\bv(c)t) \ = \ g(a\otimes \bv(c)t)
 \]
 for all $c\in \bk$, $a\in A$ and $t\in\bT$. Clearly, $g$ is multiplicative with $g(0)=0$ and $g(1)=1$. Since the semiring $(B\otimes_\bk\bT)^+$ is freely generated by the underlying monoid with zero $(B\times_\bk\bT)^\bullet$, the monoid morphism $g$ extends uniquely to a semiring morphism $g^+:(B\otimes_\bk\bT)^+\to\bT^+$. Thus we are left with showing that $g^+$ is order-preserving in order to show that $g$ is a morphism of ordered blueprints. This can be verified on the generators of the partial order of $B\otimes_\bk\bT$, which stem from the relations in $B$ and $\bT$.
 
 Since $B$ is monomial, its partial order is generated by relations of the form $c_aa\leq\sum c_jb_j$ with $a,b_j\in A$ and $c_a,c_j\in k$. Note that this corresponds to the relation $c_aa\otimes 1\leq \sum c_jb_j\otimes1$ of $B\otimes_\bk\bT$. The relation $c_aa\leq\sum c_jb_j$ implies that $c_aa\equiv\sum c_jb_j\pmod{I}$ as elements of $k[A]$ where $I=\ker(\iota^\sharp)$ is the ideal of definition for $\iota:X\to T$. In other words, $p=-c_aa+\sum c_jb_j$ is an element of $I$. Since $f\in Z(p^\trop)$, we conclude that the maximum occurs twice in $\{v(c_a)f(a),v(c_j)f(b_j)\}_j$. By Lemma \ref{lemma: characterization of monomial relations in the tropical hyperfield}, this means that $g(a\otimes1)=c_af(a)\leq \sum c_jf(b_j)=\sum g(c_jb_j\otimes 1)$ in $\bT$. 
 
 Since the composition $\bT\to B\otimes_\bk\bT\to\bT$ is the identity on $\bT$, it is clear that $g$ preserves all relations of $B\otimes_\bk\bT$ coming from $\bT$ and that $g$ is $\bT$-linear. This shows that $g=\Phi^\trop(f)$ is a $\bT$-linear morphism and thus an element of $\Trop_\bv(Y)(\bT)$.
 
 The inverse map $\Psi^\trop:\Trop_\bv(Y)(\bT)\to X^\trop$ maps a $\bT$-linear morphism $g:B\otimes_\bk\bT\to\bT$ to the map $f=\Psi^\trop(g):A\to\R_{\geq0}$ that is defined by $f(a)=g(a\otimes1)$. The map $f$ is clearly multiplicative with $f(1)=1$ and thus an element of $\R_{\geq0}^A=\Hom(A,\R_{\geq0})$. 
 
 In order to show that $f$ lies indeed in the subset $X^\trop$ of $\R_{\geq0}^A$, consider an element $p=\sum c_aa$ in the ideal of definition $I$ where $c_a\in k$ and $a\in A$. Then we have $0\leq \sum c_aa$ in $B$ and thus $0\leq \sum c_aa\otimes 1$ in $B\otimes_\bk\bT$. Note that $c_aa\otimes 1=a\otimes\bv(c_a)$. Since $g$ is $\bT$-linear, we obtain
 \[
  0 \ \leq \ \sum g(c_aa\otimes 1) \ = \ \sum \bv(c_a)g(a\otimes 1)
 \]
 in $\bT$. By Lemma \ref{lemma: characterization of monomial relations in the tropical hyperfield}, this means that the maximum occurs twice among the elements $\{\bv(c_a)g(a\otimes 1)\}_{a\in A}$. Since $v(c_a)f(a)$ corresponds to $\bv(c_a)g(a\otimes 1)$ under the identification $\R_{\geq0}=\bT$, this shows that $f$ is indeed an element of $X^\trop$.
 
 It is evident that the maps $\Phi^\trop$ and $\Psi^\trop$ are mutually inverse bijections. In order to show that $\Phi^\trop$ is a homeomorphism, we begin with a reduction step to the case of a toric variety $X=T$.
 
 Recall that the topology of $X^\trop$ is by definition the subspace topology induced from $\R_{\geq0}^A$. Since the quotient map $\bk[A]\to B$ is surjective, the induced map $\bk[A]\otimes_\bk\bT\to B\otimes_\bk\bT$ is surjective as well, which means that the morphism $\Trop_\bv(Y)\to \Trop_\bv(\Spec \bk[A])$ is a closed immersion of ordered blue schemes. By Theorem \ref{thm: properties of rational point sets} \ref{T3}, $\Trop_\bv(Y)(\bT)$ is closed topological subspace $\Trop_\bv(\Spec \bk[A])(\bT)$. Thus it suffices to prove that $\Phi^\trop$ is a homeomorphism for $X=T$ and $B=\bk[A]$. This can be verified on generators of the respective topologies of $X^\trop=\R_{\geq0}^A$ and $\Trop_\bv(Y)(\bT)$.
 
 The topology of $\R_{\geq0}^A$ is generated by open subsets of the form $U_{a,V}=\{f\in T^A|f(a)\in V\}$ where $a\in A$ and $V\subset \R_{\geq0}$ is open. We have that
 \begin{multline*}
  \Phi^\trop(U_{a,V}) \ = \ \big\{ \, \Phi^\trop(f) \,\big| \, f\in \R_{\geq0}^A\text{ with }f(a)\in V \, \big\} \\
  = \ \big\{ \, g\in \Hom_\bT(B\otimes_\bk\bT,\bT)  \, \big| \, g(a\otimes 1)\in V \, \big\} \ = \ U_{a\otimes 1,V},
 \end{multline*}
 which is an open subset of $\Trop_\bv(Y)(\bT)=\Hom_\bT(B\otimes_\bk\bT,\bT)$ where we identify $V$ with the corresponding subset of $\bT=\R_{\geq0}$.
 
 Conversely, the topology of $\Trop_\bv(Y)(\bT)$ is generated by open subsets of the form $U_{ca\otimes t,V}=\{g:B\otimes_\bk\bT\to\bT|g(ca\otimes t)\in V\}$ with $c\in\bk$, $a\in A$, $t\in T$ and $V\subset\bT$ open. We have $ca\otimes t=a\otimes \bv(c)t$. Since the multiplication of $\bT$ is continuous, the subset $V_{\bv(c)t}=\{a\in\bT|\bv(c)ta\in V\}$ is open as well. Since $g(a\otimes \bv(c)t)=\bv(c)tg(a\otimes 1)\in V$ if and only if $g(a\otimes 1)\in V_{\bv(c)t}$, we find that
 \begin{multline*}
  \Psi^\trop(U_{ca\otimes t,V}) \ = \ \big\{ \, \Psi^\trop(g) \,\big| \, g\in \Hom_\bT(B\otimes_\bk\bT,\bT)\text{ with }g(ca\otimes t)\in V \, \big\} \\
  = \ \big\{ \, f\in\R_{\geq0}^A \, \big| \, f(a)\in V_{\bv(c)t} \, \big\} \ = \ U_{a,V_{\bv(c)t}},
 \end{multline*}
 which is an open subset in $\R_{\geq0}^A$ where we identify $V_{\bv(c)t}$ with the corresponding open subset of $\R_{\geq0}=\bT$. This completes the proof that $\Phi^\trop:X^\trop\to\Trop_\bv(Y)(\bT)$ is a homeomorphism.
 
 The last step of the proof concerns the commutativity of the diagram
 \[
  \begin{tikzcd}[column sep=3cm]
   X^\an \ar{r}{\trop} \ar{d}[left]{\Phi^\an} & X^\trop \ar{d}[right]{\Phi^\trop} \\
   \Trop_\bv(\bX)(\bT) \ar{r}{\beta_\ast} & \Trop_{\bv}(Y)(\bT).
  \end{tikzcd}
 \]
 Consider an element of $X^\an$, which is a seminorm $w:R\to \R_{\geq0}$ that extends $v:k\to\R_{\geq0}$. Then $\trop(w)=w\circ\iota^\sharp_A$ and $\Phi^\trop(\trop(w)):B\otimes_\bk\bT\to\bT$ maps $ca\otimes t$ to $\bv(c)w(\iota^\sharp_A(a))t$, which is the same as $w(\iota^\sharp(ca))t$.
 
 On the other side, $\Phi^\an(w):\bR\otimes_\bk\bT\to\bT$ maps $r\otimes t$ to $w(r)t$. The image of $\Phi^\an(w)$ under $\beta_\ast$ is the morphism $B\otimes_\bk\bT\to\bT$ that maps $ca\otimes t$ to $w(\iota^\sharp(ca))t$. This shows that $\Phi^\trop(\trop(w))=\beta_\ast(\Phi^\an(w))$ and that the diagram commutes.
\end{proof}

The following description of $B\otimes_\bk\bT$ is useful for the calculation of explicit examples.

\begin{lemma}\label{lemma: the base change to the tropical hyperfield as quotient of a free algebra}
 The association $ca\otimes t\mapsto \big(v(c)t\big)\cdot a$ defines an $\bT$-linear isomorphism 
 \[\textstyle
  f: \quad B\otimes_\bk\bT \ \stackrel\sim\longrightarrow \ \bigbpgenquot{\bT}{v(c_a)a\leq\sum v(c_j)b_j \, \big| \, \pi(c_aa)=\sum \pi(c_jb_j) \text{ in }R}
 \]
 of ordered blueprints where $c_a,c_j\in \bk$ and $a,b_j\in A$.
\end{lemma}

\begin{proof}
 Let us define for the sake of this proof 
 \[
  C \ = \ \bpgenquot{\bT}{v(c_a)a\leq\sum v(c_j)b_j | \pi(c_aa)=\sum \pi(c_jb_j) \text{ in }R}.
 \]
 As a first step, we observe that 
 \[
  f(ca\otimes t) \ = \ \big(v(c)t\big)\cdot a \ = \ a\otimes v(c)t,
 \]
 which shows the independence under the action of $\bk$ on tensors. Since $\bT\to B\otimes_\bk\bT\to \bT[A]$ is the identity on $\bT$, it is clear that all relation coming from $\bT$ are preserved and that $f$ is $\bT$-linear. Thus we are left with showing that $f$ preserves the relations coming from $B$, which can be verified on generators, which are of the form $c_aa\otimes 1\leq \sum c_jb_j\otimes 1$ where $\pi(c_aa)=\sum \pi(c_jb_j)$ in $R$. This means that
 \[
  f(c_aa\otimes 1) \ = \ v(c_a)a \ \leq \ \sum v(c_j)b_j \ = \ f(c_jb_j\otimes 1),
 \]
 as desired. 

 In order to show that $f$ is an isomorphism, we will prove that the association $ta\mapsto a\otimes t$ with $t\in\bT$ and $a\in A$ defines a $\bT$-linear morphism $g:C\to B\otimes_\bk\bT$, which is obviously inverse to $f$. It is clear that this association defines a $\bT$-linear multiplicative map $\bT[A]\to B\otimes_\bk\bT$. We are left with showing that it preserves the additive relations, which can be verified on generators of the form $v(c_a)a\leq\sum v(c_j)b_j$ for which $\pi(c_aa)=\sum \pi(c_jb_j)$ in $R$. By the definition of $B$, this means that $c_aa\leq \sum c_jb_j$ in $B$ and thus
 \[
  a\otimes v(c_a) \ = \ c_aa\otimes 1 \ \leq\ \sum c_jb_j\otimes 1 \ = \ \sum b_j\otimes v(c_j).
 \]
 This concludes the proof of the lemma.
\end{proof}

\begin{ex}\label{ex: the Kajiwara-Payne tropicalization from the scheme theoretic tropicalization}
 We explain Theorem \ref{thm: recovering the KP tropicalization from the scheme theoretic tropicalization} in the case of the standard plane tropical line from Example \ref{ex: Kajiwara-Payne tropicalization}.
 
 Using the context of Theorem \ref{thm: recovering the KP tropicalization from the scheme theoretic tropicalization}, let $T=\A^2_k=\Spec k[T_1,T_2]$ be the affine plane over $k$ and $X=\Spec k[T_1,T_2]/(T_1+T_2+1)$, together with its natural closed embedding $\iota:X\to T$ as the zero set of $T_1+T_2+1$. Note that here $A=\{T_1^{e_1}T_2^{e_2}|e_1,e_2\in\N\}$.
 
 In this case, we have that the partial order of 
 \[\textstyle
  B \ = \ \bpgenquot{\bk[A]}{c_aa\leq \sum c_jb_j|\pi(c_aa)\leq\sum \pi(c_jb_j)\text{ in }\bR},
 \]
 is generated by the relations
 \[
  0 \ \leq \ T_1+T_2+1, \quad -T_1 \ \leq \ T_2+1, \quad -T_2 \ \leq \ T_1+1 \quad \text{and} \quad -1 \ \leq \ T_1+T_2.
 \]
 Using Lemma \ref{lemma: the base change to the tropical hyperfield as quotient of a free algebra} and the fact that $v(-1)=1$, we derive an isomorphism
 \[
  B\otimes_\bk\bT \ \simeq \ \bigbpgenquot{\bT[T_1,T_2]}{0\leq T_1+T_2+1,\  T_1 \leq T_2+1,\  T_2 \leq T_1+1,\  1 \leq T_1+T_2}.
 \]
 Sending a $\bT$-linear morphism $f:B\otimes_\bk\bT\to\bT$ to $\big(f(T_1\otimes 1),f(T_2\otimes 1)\big)\in\bT^2$ yields thus an identification
 \begin{multline*}
  \Trop_\bv(\bX)(\bT) \ = \ \Hom_\bT(B\otimes_\bk\bT,\bT) \\ 
  = \ \big\{\, (a_1,a_2)\in\bT^2 \, \big| \, 0\leq a_1+a_2+1,\, a_1 \leq a_2+1,\, a_2 \leq a_1+1,\, 1 \leq a_1+a_2 \, \big\}.
 \end{multline*}
 By Lemma \ref{lemma: characterization of monomial relations in the tropical hyperfield}, each of the four relations on $a_1$ and $a_2$ is equivalent to the condition that the maximum among $a_1$, $a_2$ and $1$ occurs twice. This equals the bend locus of $T_1+T_2+1$, which is exactly $X^\trop$, cf.\ Examples \ref{ex: topologies of tropical plane and tropical plane line} and \ref{ex: Kajiwara-Payne tropicalization}.
\end{ex}


\section{The relation between the tropical hyperfield and the tropical semifield}
\label{section: the relation between the tropical hyperfield and the tropical semifield}

In this section, we will explain how to recover the tropical semifield from the tropical hyperfield using some functorial constructions for ordered blueprints. The precise relation is formulated in Proposition \ref{prop: projection from the tropical hyperfield to the totally positive tropical semifield}, which is a key fact for our description of the Giansiracusa bend in terms of the scheme theoretic tropicalization in section \ref{section: recovering the bend}.

All except for Proposition \ref{prop: projection from the tropical hyperfield to the totally positive tropical semifield} is covered already in \cite{Lorscheid15} and \cite{Lorscheid18}, but we present an independent and streamlined exposition, including a shortened proof of Lemma \ref{lemma: an idempotent blueprint is isomorphic to the core of the associated totally positive ordered blueprint}.


\subsection{The tropical semifield as an algebraic blueprint}
\label{subsection: the tropical semifield as an algebraic blueprint}

Let $\overline\R=\R_{\geq0}$ be the \emph{tropical semifield} whose addition is characterized by the rule $a+b=\max\{a,b\}$ and whose multiplication is the usual multiplication of real numbers. We associate with $\overline\R$ the algebraic blueprint $\T=(\overline\R,\overline\R,=)$.


\subsection{Idempotent ordered blueprints}
\label{subsection: idempotent ordered blueprints}

An ordered blueprint $B$ is \emph{idempotent} if $B^+$ is an idempotent semiring, i.e.\ if $1+1=1$. Given an ordered blueprint $B$, we define its \emph{associated idempotent ordered blueprint} as $B^\idem=\bpgenquot{B}{1+1\=1}$. The quotient map defines a canonical morphism $B\to B^\idem$.

The construction of $B^\idem$ is functorial. Given a morphism $f:B\to C$ of ordered blueprints, the composition $B\to C\to C^\idem$ factors uniquely through a morphism $f^\idem:B^\idem\to C^\idem$ by the universal property of the quotient map $B\to B^\idem$; cf.\ section \ref{subsection: quotients by relations}.

Our main example of an idempotent ordered blueprint is $\T$. Note that $\B=\F_1^\idem$ and thus $B^\idem=B\otimes_\Fun\B$ for every ordered blueprint $B$, which reinforces the functoriality of $B\mapsto B^\idem$.


\subsection{Totally positive ordered blueprints}
\label{subsection: totally positive ordered blueprints}

An ordered blueprint $B$ is \emph{totally positive} if $0\leq 1$ holds in $B$. Multiplying this relation by $x$ yields $0\leq x$ for every $x\in B^+$. Given an ordered blueprint $B$, we define its \emph{associated totally positive blueprint} as $B^\pos=\bpgenquot B{0\leq 1}$. The quotient map defines a canonical morphism $B\to B^\pos$.

The construction of $B^\pos$ is functorial: given a morphism $f:B\to C$ of ordered blueprints, the composition $B\to C\to C^\pos$ factors uniquely through a morphism $f^\pos:B^\pos\to C^\pos$ by the universal property of the quotient map $B\to B^\pos$; cf.\ section \ref{subsection: quotients by relations}.

Looking back at Example \ref{ex: quotients of ordered blueprints}, we see that $\F_1^\pos=\bpgenquot\Fun{0\leq 1}$ is indeed the totally positive ordered blueprint associated with $\Fun$. More generally, we have $B^\pos=B\otimes_\Fun\Fun^\pos$ for every ordered blueprint $B$, which reinforces the functoriality of $B\mapsto B^\idem$.

\begin{lemma}\label{lemma: an idempotent blueprint is isomorphic to the core of the associated totally positive ordered blueprint}
 If $B$ is algebraic and idempotent, then $B\to B^\pos$ is a bijection between the respective underlying monoids and $x\leq y$ in $B^\pos$ if and only if $x+y=y$ in $B^+$. The canonical morphism $B=B^\core\to(B^\pos)^\core$ is an isomorphism of ordered blueprints.
\end{lemma}

\begin{proof}
 This follows at once from Proposition 2.12 and Lemma 2.15 in \cite{Lorscheid15}. For completeness, we give an independent proof in the following. 
 
 Let $\leq$ be the relation on $B^+$ that is defined by the rule that $x\leq y$ if and only if $x+y=y$. We claim that $\leq$ is an additive and multiplicative partial order on $B^+$.

 The relation $\leq$ is reflexive since the idempotent relation $x+x=x$ implies $x\leq x$. It is antisymmetric since the relations $x\leq y$ and $y\leq x$ imply $x=x+y=y$. It is transitive since if $x\leq y$ and $y\leq x$, then $x+y=y$ and $y+z=z$, and thus $x+z=x+y+z=y+z=z$, which yields $x\leq z$ as desired. To prove additivity and multiplicativity, consider $x\leq y$, i.e.\ $x+y=y$, and let $z\in B^+$. Then $(x+z)+(y+z)=x+y+z=y+z$ and $xz+yz=yz$, and henceforth $x+z\leq y+z$ and $xz\leq yz$, as claimed. This completes the proof that $\leq$ is an additive and multiplicative partial order on $B^+$.
 
 Since $0+1=1$, the relation $\leq$ contains $0\leq 1$. On the other hand, $\leq$ is generated by the relation $0\leq1$ for the following reason. Consider an equality $x+y=y$. Multiplying $0\leq 1$ by $y$ yields $0\leq y$, and therefore $x=x+0\leq x+y=y$, as claimed. 
 
 We conclude that $\leq$ is the smallest additive and multiplicative partial order on $B^+$ that contains $0\leq 1$. This means that $\leq$ is the partial order of $B^\pos$, i.e.\
 \[
  B^\pos \ = \ \bpgenquot B{0\leq1} \ = \ (B^\bullet, B^+,\leq).
 \]
 From this it is obvious that the quotient map $B\to B^\pos$ is a bijection between the respective underlying monoids and that $(B^\pos)^\core=(B^\bullet, B^+,=)$ is equal to $B$, i.e.\ the canonical morphism $B=B^\core\to(B^\pos)^\core$ is an isomorphism.
\end{proof}

\begin{ex}\label{ex: partial order of T-pos}
 As an immediate consequence of the characterization of the partial order of $B^\pos$ for algebraic and idempotent $B$, we see that the partial order of $\T^\pos$ is the natural linear order of the nonnegative real numbers. Moreover, we have $(\T^\pos)^\core=\T$.
\end{ex}


\subsection{Recovering the tropical semifield from the tropical hyperfield}
\label{subsection: recovering the tropical semifield from the tropical hyperfield}

In this section, we explain the relation between the tropical hyperfield $\bT$ and the tropical semifield $\T$. In particular, we see that $\T$ results from $\bT$ under a functorial construction. In this sense, we can consider $\bT$ as a refinement of $\T$.

\begin{prop}\label{prop: projection from the tropical hyperfield to the totally positive tropical semifield}
 The identity map $\bT=\R_{\geq0}\to \R_{\geq0}=\T^\pos$ is a morphism $\pi:\bT\to\T^\pos$ of ordered blueprints, which induces an isomorphism $\bT^\idem\to\T^\pos$. Taking algebraic cores yields an isomorphism $(\bT^\idem)^\core\to\T$.
\end{prop}

\begin{proof}
 We begin with the verification that the identity map is a morphism $\pi:\bT\to\T^\pos$. Clearly $\pi$ is multiplicative and $\pi(0)=0$ and $\pi(1)=1$. Since $\bT^+=\N[\R_{>0}]$ is freely generated by the monoid $\R_{>0}=\R_{\geq0}-\{0\}$ and $\pi(0)=0$, it extends uniquely to a semiring homomorphism $\pi^+:\bT^+\to \T^+$. 
 
 That $\pi^+$ is order-preserving can be verified on the generators of the partial order $\leq$ of $\bT$, which are of the form $a\leq a+b$ and $b\leq a+a$ where $a,b\in\bT$ with $a$ larger than $b$ (as real numbers). Then we have $a+b=a$ in $\T$ and thus $\pi(a)\leq\pi(a)+\pi(b)$ in $\T^\pos$, which is the first desired relation. As we have seen in Example \ref{ex: partial order of T-pos}, we have $b\leq a$ in $\T^\pos$, and multiplying $0\leq1$ by $a$ yields $0\leq a$ in $\T^\pos$. Thus $b=b+0\leq a+a$ in $\T^\pos$, which is the second desired relation. This concludes the proof that $\pi:\bT\to \T^\pos$ is a morphism.
 
 The morphism $\pi:\bT\to \T^\pos$ induces a morphism $\bar\pi:\bT^\idem\to (\T^\pos)^\idem=\T^\pos$ where we use in the latter identification that $\T^\pos$ is already idempotent. Since the composition of $\bar\pi$ with the surjective quotient map $\bT\to\bT^\idem$ is the bijective map $\pi$, we conclude that $\bar\pi$ is also bijective. Thus it suffices to show that $\bar\pi^+:(\bT^\idem)^+\to(\T^\pos)^+=\T^+$ is a semiring isomorphism and that the defining relation $0\leq1$ of $\T^\pos=\bpgenquot{\T}{0\leq1}$ occurs in $\bT^\idem$.
 
 We begin with the proof that the surjective semiring homomorphism $\bar\pi^+$ is an isomorphism. It suffices to show that every equality $a+b=a$ in $\T^+$ holds already in $(\bT^\idem)^+$. This is so since $a+b=a$ in $\T^+$ means that $a$ is larger than $b$ (as real numbers) and thus $a\leq a+b$ in $\bT^+$. On the other side, we have $b\leq a+a$ in $\bT$ and thus $a+b\leq a+a+a=a$ in $\bT^\idem$. This shows that $a+b=a$ in $\bT^\idem$, as desired. 
 
 Finally, we observe that $0\leq 1+1=1$ in $\bT^\idem$, which concludes the proof that $\bar\pi:\bT^\idem\to\T^\pos$ is an isomorphism of ordered blueprints. The last claim $(\bT^\idem)^\core\simeq \T$ of the proposition follows at once from Lemma \ref{lemma: an idempotent blueprint is isomorphic to the core of the associated totally positive ordered blueprint} and Example \ref{ex: partial order of T-pos}.
\end{proof}


\section{Recovering the Giansiracusa bend}
\label{section: recovering the bend}

In their paper \cite{Giansiracusa-Giansiracusa16} on tropical scheme theory, Jeff and Noah Giansiracusa introduce the bend relations for a tropical variety. This approach finds a refinement in the author's paper \cite{Lorscheid15} that is based on ordered blueprints. We will see in this section that the tropicalization over the tropical hyperfield is a further refinement of the Giansiracusa tropicalization.

\subsection{The Giansiracusa bend}
\label{subsection: the bend}

We begin with a review of the Giansiracusa bend as a semiring. Let $\overline\R$ be the tropical semifield, as defined in section \ref{subsection: the tropical semifield as an algebraic blueprint}.

Let $k$ be a field with nonarchimedean absolute value $v:k\to\R_{\geq0}$, which we consider as a map into $\overline\R=\R_{\geq0}$ in the following. Let $X=\Spec R$ be an affine $k$-scheme and $\iota:X\to T$ a $k$-linear closed immersion into an affine $k$-scheme of the form $T=\Spec k[A]$ for some commutative monoid $A$. Let $\iota^\sharp:k[A]\to R$ be the corresponding surjection of $k$-algebras and $I=\ker\iota^\sharp$ the ideal of definition of $X$ inside $T$.

In the definition of the Giansiracusa bend, we make use of congruences for semirings, which are additive and multiplicative equivalence relations. Congruences can be characterized as those equivalence relations on semirings for which addition and multiplication of representatives defines a semiring structure on the quotient set. For more details, we refer to \cite[section 2.4]{Lorscheid18}.

The \emph{Giansiracusa bend of $R$ (with respect to $v$ and $\iota$)} is the semiring 
\[
 \Bend_{v,\iota}^{GG}(R) \ = \ \overline\R[A]/\bend_{v,\iota}(R)
\]
where $\bend_{v,\iota}(R)$ is the congruence on $\overline\R[A]$ that is generated by the \emph{bend relations}, which are relations of the form
\[\textstyle
 v(c_a)a+\sum v(c_j)b_j \ \sim \ \sum v(c_j)b_j 
\]
for which $c_aa-\sum c_jb_j\in I$ where $c_a,c_j\in k$ and $a,b_j\in A$. 

\begin{ex}\label{ex: the Giansiracusa bend}
 As an illustration, we calculate the Giansiracusa bend of the zero set of $T_1+T_2+1$. In this case, $T=\A^2_k=\Spec k[T+1,T+2]$ is the affine $k$-plane and $X$ the closed subscheme with coordinate algebra $R=k[T_1,T_2]/(T_1+T_2+1)$, together with the natural closed immersion $\iota:X\to T$. 
 
 It is easy to verify that the bend relation $\bend_{v,\iota}(R)$ of the Giansiracusa bend
 \[
  \Bend_{v,\iota}^{GG}(R) \ = \ \overline\R[T_1,T_2]/\bend_{v,\iota}(R)
 \]
 is generated by the relations
 \[
  T_1+T_2+1 \ \sim \ T_1+T_2 \ \sim \ T_1+1 \ \sim \ T_2+1.
 \]
\end{ex}

\begin{rem}
 To simplify our exposition, we omit the geometric counterpart that makes use of semiring schemes, but turn to the description of the bend as a blueprint right away. For details on a geometric description of the Giansiracusa bend as a semiring scheme, cf.\ \cite{Giansiracusa-Giansiracusa16} and \cite{Lorscheid15}.
\end{rem}


\subsection{The bend as a blueprint}
\label{subsection: the bend as a blueprint}

The description of the Giansiracusa bend as a quotient of $\overline\R[A]$ carries additional information that can be captured in terms of an algebraic blueprint $\Bend_{v,\iota}(R)$ whose ambient semiring is $\Bend_{v,\iota}(R)^+=\Bend_{v,\iota}^{GG}(R)$ and whose underlying monoid is
\[
 \Bend_{v,\iota}(R)^\bullet \ = \ \big\{ \, ta \in \overline\R[A] \, \big| \, t\in \overline\R, a\in A \, \big\}.
\]
We call the blueprint $\Bend_{v,\iota}(R)$ the \emph{bend of $R$ (with respect to $v$ and $\iota$)}. As explained in section \ref{subsection: free algebras}, the underlying monoid of a free algebra over an ordered blueprint consist of monomials, which yields the description
\[
 \Bend_{v,\iota}(R) \ = \ \bpquot{\T[A]}{\bend_{v,\iota}(R)}
\]
of the bend of $R$ where $\T=(\overline\R,\overline\R,=)$ is as in section \ref{subsection: the tropical semifield as an algebraic blueprint}.

The association $t\mapsto t1$ defines a morphism $\T\to\Bend_{v.\iota}(R)$ of blueprints, which turns the bend of $R$ into a blue $\T$-algebra.

We define the \emph{bend of $X$ (with respect to $v$ and $\iota$)} as the blue $\T$-scheme $\Bend_{v,\iota}(X)=\Spec\big(\Bend_{v,\iota}(R)\big)$.

\begin{ex}\label{ex: the bend as a blueprint}
 We continue Example \ref{ex: the Giansiracusa bend} where we described the Giansiracusa bend of the zero set of $T_1+T_2+1$. The bend of $R$ is the blueprint
 \[
  \Bend_{v,\iota}(R) \ = \ \bigbpgenquot{\T[T_1,T_2]}{T_1+T_2+1\= T_1+T_2\= T_1+1\= T_2+1}
 \]
 where $x\= y$ stands for $x\leq y$ and $y\leq x$. More explicitly, the ambient semiring of $\Bend_{v,\iota}(R)$ is 
 \[
  \Bend_{v,\iota}(R)^+ = \Bend_{v,\iota}^{GG}(R) = \overline\R[T_1,T_2] \big/ \big\langle T_1+T_2+1\sim T_1+T_2\sim T_1+1\sim T_2+1 \big\rangle
 \]
 and its underlying monoid is
 \[
  \Bend_{v,\iota}(R)^\bullet \ = \ \big\{ \, [tT_1^{e_1}T_2^{e_2}] \in \Bend_{v,\iota}^{GG}(R) \, \big| \, t\in\overline{\R}, e_1,e_2\in\N \, \big\}.
 \]
\end{ex}


\subsection{The Kajiwara-Payne tropicalization from the bend}
\label{subsection: the Kajiwara-Payne tropicalization from the bend}

The key insight from \cite{Giansiracusa-Giansiracusa16} is that the bend of $X$ with respect to a closed immersion $\iota:X\to T$ into a toric $k$-variety $X$ recovers tropicalization $X^\trop$ of $X$. At the same time we can recover the analytification $X^\an$ of $X$ from the bend with respect to the closed immersion $\widehat\iota:X\to \widehat T$ that is induced by the surjection $\pi:k[R^\bullet]\to R$ where $\widehat T=\Spec k[R^\bullet]$, $R^\bullet$ is the underlying monoid of $R$ and $\pi$ sends a formal linear combination $\sum c_aa$ of elements of $R$ to its value as an element of $R$.

Similarly to the case of $\bT$-rational points of an ordered blue $\bT$-scheme, we can endow the set $Y(\T)=\Hom_\T(S,\T)$ of $\T$-rational points of a blue $\T$-scheme $Y=\Spec S$ with a the compact-open topology with respect to the discrete topology of $S$ and the Euclidean topology of $\T=\R_{\geq0}$. The following is Theorem 9.1 in \cite{Lorscheid15}; also cf.\ \cite{Giansiracusa-Giansiracusa16}.

\begin{thm}\label{thm: KP-tropicalization from GG-bend}
 The Berkovich space $X^\an$ is naturally homeomorphic to $\Bend_{v,\widehat\iota}(X)(\T)$, the Kajiwara-Payne tropicalization $X^\trop$ is naturally homeomorphic to $\Bend_{v,\iota}(X)(\T)$ and the diagram 
 \[
  \begin{tikzcd}[column sep=3cm]
   X^\an \ar{r}{\trop} \ar{d}[left]{\simeq} & X^\trop \ar{d}[right]{\simeq} \\
   \Bend_{v,\widehat\iota}(X)(\T) \ar{r}{\beta_\ast} & \Bend_{v,\iota}(X)(\T).
  \end{tikzcd}
 \]
 of continuous maps commutes.
\end{thm}

\begin{ex}\label{ex: Kajiwara-Payne tropicalization from the bend}
 We continue Examples \ref{ex: the Giansiracusa bend} and \ref{ex: the bend as a blueprint} where we consider the bend
 \[
  \Bend_{v,\iota}(R) \ = \ \bigbpgenquot{\T[T_1,T_2]}{T_1+T_2+1\= T_1+T_2\= T_1+1\= T_2+1}
 \]
 of $R=k[T_1,T_2]/(T_1+T_2+1)$. Let $X=\Spec \big(\Bend_{v,\iota}(R)\big)$. Sending a $\T$-linear morphism $f:\Bend_{v,\iota}(R)$ to $\big(f(T_1),f(T_2)\big)\in\R_{\geq0}^2$ defines an identification
 \begin{multline*}
  X(\T) \ = \ \Hom_\T(\Bend_{v,\iota}(R),\T\big) \\
  = \ \big\{ \, (a_1,a_2)\in\R_{\geq0}^2\, \big| \, \max\{a_1,a_2,1\}=\max\{a_1,a_2\}=\max\{a_1,1\}=\max\{a_2,1\}\, \big\}.
 \end{multline*}
 The condition 
 \[
  \max\{a_1,a_2,1\} \ = \ \max\{a_1,a_2\} \ = \ \max\{a_1,1\} \ = \ \max\{a_2,1\}
 \]
 is satisfied precisely for those $(a_1,a_2)\in\R_{\geq0}^2$ for which the maximum occurs twice among $a_1$, $a_2$ and $1$. This set is the bend locus of $T_1+T_2+1$, which equals $X^\trop$; also cf. Example \ref{ex: the Kajiwara-Payne tropicalization from the scheme theoretic tropicalization}.
\end{ex}


\subsection{Recovering the bend from the scheme theoretic tropicalization}
\label{subsection: recovering the bend from the scheme theoretic tropicalization}

In this section, we do not require any assumptions on the monoid $A$. Therefore Theorem \ref{thm: recovering the bend relation} applies to both the analytification, in which case we use the closed immersion $\iota:\Spec R\to\Spec k[R^\bullet]$ from section \ref{subsection: the Kajiwara-Payne tropicalization from the bend}, and the tropicalization, in which case we use a closed immersion $\iota:X\to T$ into a toric variety $T$.

Let $\bk=k^\mon$ and $\bv:\bk\to\bT$ the associated morphism from Theorem \ref{thm: nonarchimedean seminoms as morphisms}. Let $\bR=R^\mon$ and $\pi:\bk[A]\to\bR$ the morphism that maps $ca$ to $\iota^\sharp(ca)$. As in section \ref{subsection: analytification and tropicalization as rational point sets}, we define
\[\textstyle
 B \ = \ \bpgenquot{\bk[A]}{c_aa\leq \sum c_jb_j|\pi(c_aa)\leq\sum \pi(c_jb_j) \text{ in }\bR}
\]
where $c_a,c_j\in\bk$ and $a,b_j\in A$. Let $Y=\Spec B$.

\begin{thm}\label{thm: recovering the bend relation}
 There are canonical isomorphisms 
 \[
  \Bend_{v,\iota}(X)^\pos \ \stackrel\sim\longrightarrow \ \Trop_\bv(Y)^\idem \quad \text{and} \quad \Bend_{v,\iota}(X) \ \stackrel\sim\longrightarrow \ (\Trop_\bv(Y)^\idem)^\core.
 \]
\end{thm}

\begin{proof}
 The theorem follows at once from \cite[Cor.\ 7.13]{Lorscheid15}, using Proposition \ref{prop: projection from the tropical hyperfield to the totally positive tropical semifield}. In the following, we give an independent proof.
 
 Thanks to Lemma \ref{lemma: an idempotent blueprint is isomorphic to the core of the associated totally positive ordered blueprint}, the second claim of the theorem follows from the first claim. Thus it suffices to show that the association $ta\mapsto a\otimes t$ with $t\in\T=\bT$ and $a\in A$ defines an isomorphism
 \[
  f: \quad \Bend_{v,\iota}(R)^\pos \ = \ \big( \bpquot{\T[A]}{\bend_v(B)} \big)^\pos \ \longrightarrow \ \big( B\otimes_\bk\bT \big)^\idem.
 \] 
 
 We begin with the proof that $f$ is well-defined as a morphism. Since the association $ta\mapsto a\otimes t$ defines a multiplicative map $\T[A]\to\Trop_\bv(B)$ that maps $0$ to $0$ and $1$ to $1$, we are left with showing that it respects all the defining relations of $\Bend_{v,\iota}(R)^\pos$. Since the tensor product is compatible with taking the quotient by the relation $1+1\=1$ and by Proposition \ref{prop: projection from the tropical hyperfield to the totally positive tropical semifield}, we have
 \[
  \big( B\otimes_\bk\bT \big)^\idem \ = \ B\otimes_\bk\bT^\idem \ = \ B\otimes_\bk\T^\pos.
 \]
 This shows that $f$ respects all relations coming from $\T^\pos$, which reduces our proof to the verification that the bend relations hold in $B\otimes_\bk\T^\pos$. This can be verified on generators. For $c_a,c_j\in k$ and $a,b_j\in A$, consider the relation
 \[\textstyle
  v(c_a)a+\sum v(c_j)b_j \ = \ \sum v(c_j)b_j
 \]
 in $\bend_v(B)$ stemming from the element $c_aa-\sum c_jb_j$ in the ideal of definition $I=\ker(\iota^\sharp)$ of $X$ in $T$. Then we have $c_aa\leq \sum c_jb_j$ in $B$ and 
 \[
  a\otimes v(c_a) \ = \ c_aa\otimes 1 \ \leq \ \sum c_jb_j\otimes 1 \ = \ \sum b_j\otimes v(b_j)
 \]
 in $B\otimes_\bk\bT$. Since $\T^\pos=\bT^\idem$ is idempotent, this yields 
 \[\textstyle
  a\otimes v(c_a)+\sum b_j\otimes v(c_j) \quad \leq \quad \sum b_j\otimes v(b_j)+\sum b_j\otimes v(c_j) \quad = \quad \sum b_j\otimes v(c_j)
 \]
 in $B\otimes_\bk\T^\pos$, and since $\T^\pos$ is totally positive, this yields 
 \[\textstyle
  \sum b_j\otimes v(c_j) \quad = \quad 0+\sum b_j\otimes v(c_j) \quad \leq \quad a\otimes v(c_a)+\sum b_j\otimes v(c_j) 
 \]
 in $B\otimes_\bk\T^\pos$. Thus we have $a\otimes v(c_a)+\sum b_j\otimes v(c_j) =\sum b_j\otimes (c_j)$ in $\Trop_\bv(B)$ as desired, which shows that $f:\Bend_{v,\iota}(R)^\pos\to \Trop_\bv(B)$ is a well-defined $\T^\pos$-linear morphism of ordered blueprints.

 In order to show that $f$ is an isomorphims of ordered blueprints, we need to show that all defining relations of $B\otimes_\bk\T^\pos$ are already contained in $\Bend_{v,\iota}(R)^\pos$. Since $f$ is $\T^\pos$-linear, we can restrict ourselves to the relations coming from $B$, which are of the form $c_aa\otimes 1\leq\sum c_jb_j\otimes 1$ whenever $c_aa\leq \sum c_jb_j$ in $B$. In this case, $c_aa-\sum c_jb_j$ is an element of $I$ and $\Bend_{v,\iota}(R)$ contains the bend relation $v(c_a)a+\sum v(c_j)b_j =\sum v(c_j)b_j$. Thus we have
 \[\textstyle
  v(c_a)a \quad = \quad v(c_a)a+0  \quad \leq \quad v(c_a)a+\sum v(c_j)b_j \quad = \quad \sum v(c_j)b_j
 \]
 in $\Bend_{v,\iota}(R)^\pos$, as desired. This shows that $f$ is an isomorphism and concludes the proof of the theorem.
\end{proof}

Let $R$ and $B$ be as before. As an immediate consequence of Theorem \ref{thm: recovering the bend relation}, we find the following natural interpretation of the Giansiracusa bend $\Bend_{v,\iota}^{GG}(R)=\Bend_{v,\iota}(R)^+$.

\begin{cor}\label{cor: the Giansiracusa bend from the scheme theoretic tropicalization}
 We have a canonical identification $\Bend_{v,\iota}^{GG}(R)=(B\otimes_\bk\bT^\idem)^+$. \qed
\end{cor}

\begin{ex}
 We illustrate Theorem \ref{thm: recovering the bend relation} in the example of the zero set $X$ of $T_1+T_2+1$ in the affine plane $T=\A^2_k$ over $k$. We recall the notation from Example \ref{ex: the Kajiwara-Payne tropicalization from the scheme theoretic tropicalization}: let $R=k[T_1,T_2]/(T_1+T_2+1)$ be the coordinate ring of $X$ and $\iota:X\to\A^2_k$ be the closed immersion as a subscheme. Let $\bk=k^\mon$ be the associated monomial ordered blueprint and $\bv:\bk\to\bT$ the morphism associated with the given nonarchimedean absolute value $v:k\to\R_{\geq0}$. Let 
 \[
  B \ = \ \bigbpgenquot{\bk[T_1,T_2]}{0 \leq T_1+T_2+1, \, -T_1 \leq T_2+1, \, -T_2 \leq T_1+1, \, -1 \ \leq \ T_1+T_2}
 \]
 be the ordered blueprint associates with $\iota$. As explained in Example \ref{ex: the Kajiwara-Payne tropicalization from the scheme theoretic tropicalization}, we have
 \begin{multline*}
  \Trop_\bv(B) \ = \ B\otimes_\bk\bT \\
  = \ \bigbpgenquot{\bT[T_1,T_2]}{0\leq T_1+T_2+1,\  T_1 \leq T_2+1,\  T_2 \leq T_1+1,\  1 \leq T_1+T_2}.
 \end{multline*}
 If we impose the additional relation $1+1=1$ on $\Trop_\bv(B)$, then we conclude that 
 \[
  T_1+T_2+1 \ \leq \ T_1 + T_2 + T_1+T_2 \ = \ T_1+T_2
 \]
 (using $1\leq T_1+T_2$) and that
 \[
  T_1+T_2 \ = \ T_1+T_2+0 \ \leq \ T_1+T_2+1+1 \ = \ T_1+T_2+1
 \]
 (using $0\leq 1+1$) in $\Trop_\bv(B)^\idem$. Thus we gain the equality $T_1+T_2+1=T_1+T_2$ in $(\Trop_\bv(B)^\idem)^+$. Repeating the previous argument with the roles of $T_1$, $T_2$ and $1$ exchanged yields the relations
 \[
  T_1+T_2+1 \ = \ T_1+T_2 \ = \ T_1+1 \ = \ T_2+1
 \]
 in $(\Trop_\bv(B)^\idem)^+$, which generate the bend relation $\bend_{v,\iota}(R)$ on $\overline{\R}[T_1,T_2]$, cf.\ Examples \ref{ex: the Giansiracusa bend} and \ref{ex: the bend as a blueprint}.
\end{ex}

\bibliographystyle{plain}
\bibliography{tropical}

\end{document}